\documentclass[reqno]{amsart}

\usepackage{amsmath}
\usepackage{amssymb}
\usepackage{amsthm}
\usepackage{xspace}
\usepackage{pictexwd}
\usepackage{float}  \restylefloat{figure} %\restylefloat{table}
\usepackage[OT2,OT1]{fontenc}
\usepackage{amscd}
\usepackage{enumerate}

\theoremstyle{plain}
\newtheorem{thm}{Theorem}[section]
\newtheorem*{thm*}{Theorem}
\newtheorem{prop}[thm]{Proposition}
\newtheorem{lem}[thm]{Lemma}
\newtheorem{cor}[thm]{Corollary}

\theoremstyle{definition}
\newtheorem{mydef}{Definition}[section]

\theoremstyle{remark}
\newtheorem*{rem}{Remark}

\def\forcehmode{\hskip0pt\relax}

\let\myskip=\medskip

\def\definebb#1=#2.{\def#1{{{\mathbb #2}^{\vphantom{x}}}}}
\definebb \c=C.
\definebb \r=R.
\definebb \s=S.
\definebb \z=Z.
\definebb \h=H.
\definebb \q=Q.

\def\Cal#1{\mathcal{#1}}

\def\calA{{\Cal A}}

\def\calK{{\Cal K}}

\def\srf{P}
\def\pln{U}
\def\bdl{L}
\def\autfac{(\pln,\lat,\bdl)}
\def\subgrp{{\tilde\lat}}
\def\grpsubgrp{\lat/\subgrp}
\def\tngpln{T_{\pln}}
\def\tngsrf{T_{\srf}}
\def\qtbdl{E}
\def\KX{\calK_X}

\def\cpxprj{{\mathbb C\operatorname{P}^1}}
\def\cpxpln{\c}
\def\hyppln{\hyp}

\def\tMod{\widetilde{\Mod}}
\def\Ggp{\Ga_{g;p_1,\dots,p_r}}
\def\tTgp{\tilde T_{g;p_1,\dots,p_r}}
\def\Tgp{T_{g;p_1,\dots,p_r}}
\def\Mgp{M_{g;p_1,\dots,p_r}}
\def\tModgp{\tMod_{g;p_1,\dots,p_r}}
\def\Modgp{\Mod_{g;p_1,\dots,p_r}}
\def\Modmgp{\Mod^m_{g;p_1,\dots,p_r}}
\def\Smgp{S^m}

\def\lev{{s}}
\def\levm{\lev_m}
\def\lat{\Ga}
\def\lats{\lat^*}
\def\centrgen{{u}}

\def\piorb{\pi}
\def\piorbO{\pi^0}

\def\tC{{\tilde C}}
\def\tG{{\tilde G}}

\def\Gm{{G_m}}

\def\bA{{\bar A}}
\def\bB{{\bar B}}
\def\bC{{\bar C}}

\def\bV{{\bar V}}

\def\abcd{{a\,b\choose c\,d}}

\def\hyp{\h}

\def\PSL{\MathOpPSL(2,\r)}
\def\SL{\MathOpSL(2,\r)}

\def\tPSL{\widetilde{\MathOpPSL}(2,\r)}

\def\dd{\partial}

\def\al{{\alpha}}
\def\be{{\beta}}
\def\Ga{{\Gamma}}
\def\ga{{\gamma}}
\def\de{{\delta}}
\def\De{{\Delta}}
\def\si{{\sigma}}

\def\hsi{{\hat\si}}
\def\arf{\si}

\def\st{\,\,\big|\,\,}
\def\<{\langle}
\def\>{\rangle}
\def\ie{i.e.\xspace}

\let\ge=\geqslant
\let\le=\leqslant

\DeclareMathOperator{\Aut}{Aut}
\DeclareMathOperator{\Arf}{Arf}

\DeclareMathOperator{\id}{id}
\DeclareMathOperator{\Hol}{Hol}
\DeclareMathOperator{\MathOpPSL}{PSL}
\DeclareMathOperator{\Mod}{Mod}
\DeclareMathOperator{\MathOpSL}{SL}
\DeclareMathOperator{\MathOpSp}{Sp}
\DeclareMathOperator{\ord}{ord}

\DeclareMathOperator{\PSU}{PSU}

\let\Im=\undefined \DeclareMathOperator{\Im}{Im}
 
\let\mod=\undefined \DeclareMathOperator{\mod}{mod}

\begin{document}

\author[Sergey Natanzon]{Sergey Natanzon}
\address{Moscow State University, Korp.~A, Leninske Gory, 11899 Moscow, Russia}
\address{Institute of Theoretical and Experimental Physics, Moscow, Russia}
\address{Independent University of Moscow, Bolshoi Vlasevsky Pereulok, 11 Moscow, Russia}
\email{natanzon@mccme.ru}
\author[Anna Pratoussevitch]{Anna Pratoussevitch}
\address{Department of Mathematical Sciences\\ University of Liverpool\\ Peach Street \\ Liverpool L69~7ZL}
\email{annap@liv.ac.uk}

\title[Moduli Spaces of Gorenstein Singularities]
{Topological Invariants and Moduli Spaces of Gorenstein Quasi-Homogeneous Surface Singularities}

%\title[Moduli Spaces of $\mathbb{Q}$-Gorenstein Singularities]
%{Topological Invariants and Moduli Spaces of hyperbolic $\mathbb{Q}$-Gorenstein Quasi-Homogeneous Surface Singularities}

\begin{date}  {\today} \end{date}

\thanks{Research partially supported by grants RFBR 10-01-00678-a, NSh-8462.2010.1 and SFB 611 (DFG)}
%INTAS 05-7805, RFBR-07-01-00593, NWO-RFBR 047.011.2004.026

\begin{abstract}

We describe all connected components of the space of hyperbolic Go\-ren\-stein quasi-homogeneous surface singularities.
We prove that any connected component is homeomorphic to a quotient of ${\mathbb R}^d$ by a discrete group.

%Our proof is based on the representation of the space of Gorenstein quasi-homogeneous surface singularities of level $m$
%corresponding to a certain Fuchsian group as a finite affine space of ${\mathbb Z}/m{\mathbb Z}$-valued functions on the Fuchsian group.

\end{abstract}


% AG9809138: Primary 14H10, 32G15; Secondary 32G81, 81T40, 14M30.

% AG9908085: Primary 14H10, 32G15, 81T40; Secondary 14N, 14M.

% Q-Gorenstein papers of mine: Primary 32S25; Secondary 14J17, 14J60

% 32S25: Surface and hypersurface singularities

% 30F10: Compact Riemann surfaces and uniformization

% 14J17: Singularities

% 14J60: Vector bundles on surfaces and higher-dimensional varieties, and their
% moduli

% 14H10: Families, moduli (algebraic)

% 14H15: Families, moduli (analytic)

% 14M: Special varieties

% 14N: Projective and enumerative geometry

% 32G: Deformations of analytic structures

% 32G15: Moduli of Riemann surfaces, Teichmueller theory

% 32G81: Applications to physics

% 81T40: Two-dimensional field theories, conformal field theories, etc.


\maketitle

\section{Introduction}

In this paper we study moduli spaces of hyperbolic Gorenstein quasi-homo\-ge\-ne\-ous surface singularities (GQHSS). 
A normal isolated singularity of dimension~$n$ is Gorenstein
if and only if there is a nowhere vanishing $n$-form on a punctured neighbourhood of the singular point.
GQHSS can be spherical, Euclidean or hyperbolic.
In this paper we are going to study the largest class, the class of hyperbolic GQHSS.
See a remark at the end of the paper for more information about the other two classes of GQHSS.
%There are infinitely many hyperbolic GQHSS, while the classes of spherical resp.\ Euclidean GQHSS are finite.

According to work of Dolgachev~\cite{Dolgachev:1983} hyperbolic GQHSS of level~$m$ are in 1-to-1 correspondence
with $m$-th roots of tangent bundles of Riemann orbifolds,
i.e.\ with (singular) complex line bundles on orbifolds such that their $m$-th tensor power coincides with the tangent bundle.
We find conditions for the existence of GQHSS of level~$m$ with orbifold of given signature.
We then consider the space of all GQHSS of level~$m$ with orbifolds of given signature and genus~$g\ge0$.
We show that the space is connected if $g=0$ or if $g>1$ and $m$ is odd and that the space has $2$~connected components if $g>1$ and $m$ is even.
We also determine the number of components in the case~$g=1$.
%with lifts of a Fuchsian group~$\lat$ into finite coverings~$\Gm$ of the group~$\PSL$,
%\ie subgroups $\lats$ of $\Gm$ such that the restriction of the covering map $\Gm\to\PSL$ to $\lats$ is an isomorphism $\lats\to\lat$.
%The topology of the covering spaces~$\Gm$ induces topology on the space of all lifts of Fuchsian groups of a given signature
%and hence a topology on the space of all hyperbolic Gorenstein quasi-homogeneous surface singularities.
%We assign to any hyperbolic GQHSS a set of topological invariants.
%We then show that this set of topological invariants determines a connected component of the moduli space of hyperbolic GQHSS.
Moreover we prove that any connected component is homeomorphic to a quotient of~$\r^d$ by a discrete group action.
% space of the form $\r^d/\Mod$, where $\Mod$ is a discrete group acting on~$\r^d$.

\myskip
The main technical tool is the following:
We assign (Theorem~\ref{thm-corresp}) to a hyperbolic GQHSS of level~$m$ with corresponding Fuchsian group~$\lat$
a unique function on the space of homotopy classes of simple contours on the orbifold~$P=\hyp/\lat$ with values in~$\z/m\z$,
the associated $m$-Arf function.

\myskip\noindent
The $m$-Arf functions are described by simple geometric properties:

\myskip\noindent
{\bf Definition:}
Let $P$ be a Riemann orbifold and $p\in P$.
We denote by $\piorbO(P,p)$ the set of all non-trivial elements of the orbifold fundamental group $\piorb(P,p)$ that can be represented by simple contours.
An {\it $m$-Arf function\/} is a function
$$\arf:\piorbO(P,p)\to\z/m\z$$
satisfying the following conditions
\begin{enumerate}[1.]
\item
$\arf(bab^{-1})=\arf(a)$ for any elements~$a,b\in\piorbO(P,p)$,
\item
$\arf(a^{-1})=-\arf(a)$ for any element~$a\in\piorbO(P,p)$ that is not of order~$2$,
\item
$\arf(ab)=\arf(a)+\arf(b)$
for any elements~$a$ and~$b$ which can be represented by a pair of simple contours in $P$
intersecting in exactly one point~$p$ with $\<a,b\>\ne0$,
\item
$\arf(ab)=\arf(a)+\arf(b)-1$
for any elements~$a,b\in\piorbO(P,p)$ such that the element~$ab$ is in~$\piorbO(P,p)$
and the elements~$a$ and~$b$ can be represented by a pair of simple contours in $P$
intersecting in exactly one point~$p$ with $\<a,b\>=0$
and placed in a neighbourhood of the point~$p$ as shown in Figure~\ref{fig-neg-pair}.

% Picture

\begin{figure}
  \begin{center}
    \forcehmode
      \bgroup
        \beginpicture
          \setcoordinatesystem units <25 bp,25 bp>
          \multiput {\phantom{$\bullet$}} at -2 0 2 2 /
          \plot 0 0 -2 1 /
          \arrow <7pt> [0.2,0.5] from 0 0 to -1 0.5
          \plot 0 0 -1 2 /
          \arrow <7pt> [0.2,0.5] from -1 2 to -0.5 1
          \plot 0 0 1 2 /
          \arrow <7pt> [0.2,0.5] from 0 0 to 0.5 1
          \plot 0 0 2 1 /
          \arrow <7pt> [0.2,0.5] from 2 1 to 1 0.5
          \put {$a$} [br] <0pt,2pt> at -2 1
          \put {$a$} [br] <0pt,2pt> at -1 2
          \put {$b$} [bl] <0pt,2pt> at 2 1
          \put {$b$} [bl] <0pt,2pt> at 1 2
        \endpicture
      \egroup
  \end{center}
  \caption{$\hsi(ab)=\hsi(a)+\hsi(b)-1$}
  \label{fig-neg-pair}
\end{figure}

\item
For any elliptic element~$c$ of order~$p$ we have $p\cdot\arf(c)+1\equiv0\mod\,m$.
\end{enumerate}

\myskip\noindent
In order to be able to state our main results we need to give some definitions and notation.

\myskip\noindent
{\bf Definition:}
Let $\lat$ be a Fuchsian group of signature~$(g:p_1,\dots,p_r)$ and let $P=\hyp/\lat$ be the corresponding orbifold.
Let $\arf:\piorbO(P,p)\to\z/m\z$ be an $m$-Arf function.
We define the {\it Arf invariant\/} $\de=\de(P,\arf)$ of~$\arf$ as follows:
If~$g>1$ and $m$ is even then we set $\de=0$ if there is a standard basis $\{a_1,b_1,\dots,a_g,b_g,c_{g+1},\dots,c_n\}$ of the fundamental group $\piorb(P,p)$ such that
$$\sum\limits_{i=1}^g(1-\arf(a_i))(1-\arf(b_i))\equiv0\mod~2$$
and we set $\de=1$ otherwise.
If~$g>1$ and $m$ is odd then we set $\de=0$.
If~$g=0$ then we set $\de=0$.
If~$g=1$ then we set
$$\de=\gcd(m,p_1-1,\dots,p_r-1,\arf(a_1),\arf(b_1)),$$
%$$\de=\gcd(m,\arf(a_1),\arf(b_1),\arf(c_2)+1,\dots,\arf(c_n)+1),$$
where $\{a_1,b_1,c_2,\dots,c_{r+1}\}$ is a standard basis of the fundamental group $\piorb(P,p)$.
The {\it type of the $m$-Arf function\/}~$(P,\arf)$ is the tuple $(g,p_1,\dots,p_r,\de)$,
where $\de$ is the Arf invariant of $\arf$ defined above.

\myskip\noindent
{\bf Definition:}
We denote by~$\Smgp(t)=\Smgp(g,p_1,\dots,p_r,\de)$ the set of all GQHSS of level~$m$ and signature~$(g:p_1,\dots,p_r)$
such that the associated $m$-Arf function is of type~$t=(g,p_1,\dots,p_r,\de)$.

\myskip\noindent
The following Theorem summarizes the main results:

\myskip\noindent
{\bf Theorem:}
\begin{enumerate}[1)]
\item
Two hyperbolic GQHSS are in the same connected component of the space of all hyperbolic GQHSS
if and only if they are of the same type.
In other words, the connected components of the space of all hyperbolic GQHSS are those sets $\Smgp(t)$ that are not empty.
\item
The set $\Smgp(t)$ is not empty if and only if $t=(g,p_1,\dots,p_r,\de)$ has the following properties:
\begin{enumerate}[(a)]
\item
The orders~$p_1,\dots,p_r$ are prime with~$m$ and satisfy the condition
$$(p_1\cdots p_r)\cdot\left(\sum\limits_{i=1}^r\,\frac1{p_i}-(2g-2)-r\right)\equiv0\mod m.$$
\item
If $g>1$ and $m$ is odd then $\de=0$.
\item
If $g=1$ then $\de$ is a divisor of~$\gcd(m,p_1-1,\dots,p_r-1)$.
%If $g=1$ then $\de$ is a divisor of~$m$ and~$\gcd(\{j+1\st\num^h_j+\num^p_j\ne0\})$.
\item
If $g=0$ then $\de=0$.
\end{enumerate}
\item
Any connected component~$\Smgp(t)$ of the space of all hyperbolic GQHSS of level~$m$ and signature~$(g:p_1,\dots,p_r)$ is homeomorphic to
a quotient of the space~$\r^{6g-6+2r}$ by a discrete action of a certain subgroup of the modular group (see section~\ref{components} for details).
%$$\r^{6g-6+2r}/\Modmgp(t),$$ where $\Modmgp(t)$ is a subgroup of finite index in the group~$\Modgp$ and acts discretely on $\r^{6g-6+2r}$.
\end{enumerate}

\myskip
The paper is organised as follows:
In section~\ref{sec-sing-autfac} we explore the connection between hyperbolic GQHSS,
roots of tangent bundles of orbifolds and lifts of Fuchsian groups into the coverings~$\Gm$ of~$G=\PSL$.
In section~\ref{sec-level-functions} we study algebraic properties of the covering groups~$\Gm$.
We describe a level function induced by a decomposition of the covering~$\Gm$ into sheets
and choosing a numeration of the sheets and study properties of these functions.
In section~\ref{sec-levels-on-lifts} we study lifts of Fuchsian groups into~$\Gm$.
In section~\ref{sec-m-arf} we define $m$-Arf functions.
We prove that there is a 1-1-correspondence between the set of $m$-Arf functions
and the set of functions associated to the lifts of Fuchsian groups into~$\Gm$ via the numeration of the covering sheets.
Hence these two sets are also in 1-1-correspondence with the set of all hyperbolic GQHSS of level~$m$ for a fixed orbifold.
Moreover we show in this section using the explicit description of the generalised Dehn generators of the group of homotopy classes of surface autohomeomorphisms
that the set of all $m$-Arf functions on an orbifold has a structure of an affine space.
In the last section we find topological invariants of $m$-Arf functions and prove that they describe the connected components of the moduli space.
Furthermore we show using a version of Theorem of Fricke and Klein~\cite{N1978moduli}, \cite{Zieschang:book}
that any connected component is homeomorphic to a quotient of~$\r^d$ by a discrete group action.

%$\s^1=\{z\in\c\st|z|=1\}\subset\c$, $\r_+=\{x\in\r\st x>0\}$.

\myskip
Part of this work was done during the stays at Max-Planck-Institute in Bonn and at IHES.
We are grateful to the both institutions for their hospitality and support.
We would like to thank E.B.~Vinberg and V.~Turaev for useful discussions related to this work.
%We would like to thank the referees for their valuable remarks and suggestions.

\section{Gorenstein singularities and lifts of Fuchsian groups}

\label{sec-sing-autfac}

\subsection{Singularities and automorphy factors}

\label{subsec-aut-fac}

\myskip
In this section we recall the results of  Dolgachev, Milnor, Neumann and Pinkham
\cite{Dolgachev:1975, Dolgachev:1977, Milnor:1975, Neumann:1977, Pinkham:1977}
on the graded affine coordinate rings, which correspond to quasi-ho\-mo\-ge\-neous surface singularities.

\begin{mydef}
{\it A (negative unramified) automorphy factor\/}~$\autfac$ is a complex line bundle $\bdl$ over a simply connected Riemann surface $\pln$
together with a discrete co-compact subgroup $\lat\subset\Aut(\pln)$ acting compatibly on $\pln$ and on the line bundle~$\bdl$,
such that the following two conditions are satisfied:
\begin{enumerate}[1)]
\item
The action of $\lat$ is free on $\bdl^*$, the complement of the zero-section in $\bdl$.
\item
Let $\subgrp\triangleleft\lat$ be a normal subgroup of finite index, which acts freely on $\pln$,
and let $\qtbdl\to\srf$ be the complex line bundle $\qtbdl=\bdl/\subgrp$ over the compact Riemann surface $\srf=\pln/\subgrp$.
Then $\qtbdl$ is a negative line bundle, i.e.\ the self-intersection number~$E\cdot E$ is negative.
\end{enumerate}

A simply connected Riemann surface $\pln$ can be $\cpxprj$, $\cpxpln$, or $\hyppln$, the real hyperbolic plane.
We call the corresponding automorphy factor and the corresponding singularity {\it spherical\/}, {\it Euclidean\/}, resp.\ {\it hyperbolic\/}.
\end{mydef}

\begin{rem}
There always exists a normal freely acting subgroup of $\lat$ of finite index~\cite{Dolgachev:1983}.
In the hyperbolic case the existence follows from the theorem of Fox-Bundgaard-Nielsen.
If the second assumption in the last definition holds for some normal freely acting subgroup of finite index, then it holds for any such subgroup, see~\cite{Dolgachev:1983}.
% Dolgachev 1983, p. 531, line 2.
\end{rem}

The simplest examples of such complex line bundles with group actions
are the cotangent bundle of the complex projective line $\pln=\cpxprj$ and the tangent bundle of the hyperbolic plane $\pln=\hyppln$
equipped with the canonical action of a subgroup $\lat\subset\Aut(\pln)$.

\smallskip
Let $\autfac$ be a negative unramified automorphy factor.
Since the bundle $\qtbdl=\bdl/\subgrp$ is negative,
one can contract the zero section of $\qtbdl$ to get a complex surface with one isolated singularity corresponding to the zero section.
There is a canonical action of the group $\grpsubgrp$ on this surface.
The quotient is a complex surface $X\autfac$ with an isolated singular point $o\autfac$, which depends only on the automorphy factor $\autfac$.

\smallskip
The following theorem summarizes the results of Dolgachev, Milnor, Neumann, and Pinkham:

\begin{thm}
The surface $X\autfac$ associated to a negative unramified automorphy factor $\autfac$ is a quasi-homogeneous affine algebraic surface with a normal isolated singularity.
Its affine coordinate ring  is the graded $\c$-algebra of generalised $\lat$-invariant automorphic forms
$$A=\bigoplus\limits_{m\ge0} H^0(\pln,\bdl^{-m})^{\lat}.$$
All normal isolated quasi-homogeneous surface singularities $(X,x)$ are obtained in this way,
and the automorphy factors with $(X\autfac,o\autfac)$ isomorphic to $(X,x)$ are uniquely determined by $(X,x)$ up to isomorphism.
\end{thm}

\noindent
We now recall the definition of Gorenstein singularities and the characterization of the corresponding automorphy factors.

\myskip
A normal isolated singularity of dimension~$n$ is Gorenstein
if and only if there is a nowhere vanishing $n$-form on a punctured neighbourhood of the singular point.
For example all isolated singularities of complete intersections are Gorenstein.

\smallskip
In~\cite{Dolgachev:1983} Dolgachev proved the following theorem obtained independently by W.~Neumann
(see also~\cite{Dolgachev:1983:LNMS}).
%Dolgachev proved in~\cite{Dolgachev:1983} the following characterization
%of Gorenstein quasi-homo\-ge\-ne\-ous surface singularities (GQHSS) in terms of the their automorphy factors:

\begin{thm}
\label{charact-gor}
A quasi-homogeneous surface singularity is Go\-ren\-stein
if and only if for the corresponding automorphy factor $\autfac$
there is an integer~$m$ (called the {\it level\/} or the {\it exponent\/} of the automorphy factor)
such that the $m$-th tensor power $\bdl^m$ is $\lat$-equvariantly
isomorphic to the tangent bundle $\tngpln$ of the surface $\pln$.
\end{thm}

Let $\autfac$ be an automorphy factor of level~$m$, which corresponds to a Gorenstein singularity.
The isomorphism $\bdl^m\cong\tngpln$ induces an isomorphism $\qtbdl^m\cong\tngsrf$.
%The bundle $\qtbdl$ is negative.
A simple computation with Chern numbers shows that the possible values of the exponent are $m=-1$ or $m=-2$ for $\pln=\cpxprj$,
whereas $m=0$ for $\pln=\cpxpln$ and $m$ is a positive integer for $\pln=\hyppln$.

\subsection{Hyperbolic automorphy factors and lifts of Fuchsian groups}

%\subsection{The group ${\rm PSL}(2,\mathbb R)$ and its coverings}
%\label{subsec-PSL}
%\section{Hyperbolic Gorenstein automorphy factors}
%\label{sec-hyp-gor-autfac}

We consider the universal cover $\tG=\tPSL$ of the Lie group
$$G=\PSL=\SL/\{\pm1\},$$
the group of orien\-ta\-tion-preserving isometries of the hyperbolic plane.
Here our model of the hyperbolic plane is the upper half-plane $\hyp=\{z\in\c\st\Im(z)>0\}$
and the action of an element $[\abcd]\in\PSL$ on~$\hyp$ is by
$$z\mapsto\frac{az+b}{cz+d}.$$
We denote by $[A]=[\abcd]\in\PSL$ the equivalence class of a matrix $A=\abcd\in\SL$.

\myskip
As topological space $G=\PSL$ is homeomorphic to the open solid torus $\s^1\times\c$.
%A homeomorphism $\PSL\to\s^1\times\c/\{\pm1\}$ can be given explicitly.
The fundamental group of the open solid torus $G$ is infinite cyclic.
Therefore, for each natural number $m$ there is a unique connected $m$-fold covering
$$\Gm=\tG/(m\cdot Z(\tG))$$
of $G$, where $\tG$ is the universal covering of~$G$ and $Z(\tG)$ is the centre of~$\tG$.
For $m=2$ this is the group $G_2=\SL$.

\myskip
Here is another description of the covering groups $\Gm$ of $G$ which fixes a group structure.
Let $\Hol(\hyp,\c^*)$ be the set of all holomorphic functions $\hyp\to\c^*$.

\begin{prop}
\label{fract-autforms}
The $m$-fold covering group $\Gm$ of~$G$ can be described as
$$\{(g,\de)\in G\times\Hol(\hyp,\c^*)\st \de^m(z)=g'(z)~{\rm for~all}~z\in\hyp\}$$
with multiplication $(g_2,\de_2)\cdot(g_1,\de_1)=(g_2\cdot g_1,(\de_2\circ g_1)\cdot\de_1)$.
\end{prop}

\begin{proof}
Let $X$ be the subspace of $G\times\Hol(\hyp,\c^*)$ in question.
One can check that the space $X$ is connected
and that the map $X\to G$ given by $(\ga,\de)\mapsto\ga$ is an $m$-fold covering of~$G$.
Hence the coverings $X\to G$ and $\Gm\to G$ are isomorphic.
One can check that the operation described above defines a group structure on~$X$
and that the covering map $X\to G$ is a homomorphism with respect to this group structure.
\end{proof}

\begin{rem}
This description of $\Gm$ is inspired by the notion of automorphic
differential forms of fractional degree, introduced by J.~Milnor in~\cite{Milnor:1975}.
For a more detailed discussion of this fact see~\cite{Lion:Vergne}, section~1.8.
\end{rem}

\myskip
A Fuchsian group $\lat\subset\PSL$ acts on the hyperbolic plane $\hyppln$.
{\it A hyperbolic Gorenstein automorphy factor of level~$m$\/} (associated to the Fuchsian group $\lat$)
is an action of a Fuchsian group $\lat$ on the trivial complex line bundle $\hyppln\times\c$ over the hyperbolic plane $\hyppln$ given by
$$g\cdot(z,t)=(g(z),\de_g(z)\cdot t),$$
where $\de_g:\hyppln\to\c^*$ is a holomorphic function,
for any~$g\in\lat$ we have $\de_g^m=g'$
and for any~$g_1,g_2\in\lat$ we have $\de_{g_2\cdot g_1}=(\de_{g_2}\circ g_1)\cdot\de_{g_1}$.

\begin{mydef}
A {\it lift\/} of the Fuchsian group $\lat$ into $\Gm$
is a subgroup $\lats$ of $\Gm$ such that the restriction of the covering map $\Gm\to G$ to $\lats$
is an isomorphism between $\lats$ and $\lat$.
\end{mydef}

\myskip
Using the description of the $m$-fold covering group $\Gm$ of~$G$ in Proposition~\ref{fract-autforms}
we obtain the following result:

\begin{prop}
\label{autfac-lifts}
There is a 1-1-correspondence between the lifts of $\lat$ into $\Gm$
and hyperbolic Gorenstein automorphy factors of level~$m$ associated to the Fuchsian group $\lat$.
\end{prop}

\section{Level functions on covering groups of ${\rm PSL}(2,\mathbb R)$}

\label{sec-level-functions}

%\section{Structure of the covering groups of $G={\rm PSL}(2,\mathbb R)$}

\subsection{Classification of elements in the covering groups of $G={\rm PSL}(2,\mathbb R)$}

%\subsection{Classification of elements in $G={\rm PSL}(2,\mathbb R)$}

Elements of~$G$ can be classified with respect to the fixed point behavior of their action on~$\hyp$.
An element is called {\it hyperbolic\/} if it has two fixed points,
which lie on the boundary $\dd\hyp=\r\cup\{\infty\}$ of~$\hyp$.
One of the fixed points of a hyperbolic element is attracting, the other fixed point is repelling.
%A hyperbolic element with fixed points~$\al$, $\beta$ in~$\r$ is of the form
%$$
%  \tau_{\al,\be}(\la)
%  =
%  \left[
%        \frac{1}{(\al-\be)\cdot\sqrt{\la}}\cdot\begin{pmatrix}\la\al-\be&-(\la-1)\al\be\\ \la-1&\al-\la\be\end{pmatrix}
%  \right],
%$$
%where~$\la>0$.
%A hyperbolic element with one fixed point in~$\infty$ is of the form
%% comment: \lim\limits_{\al\to\infty}\tau_{\al,\be}
%$$\tau_{\infty,\be}(\la)=\left[\frac{1}{\sqrt{\la}}\cdot\begin{pmatrix}\la&-(\la-1)\be\\ 0&1\end{pmatrix}\right]$$
%or
%% comment: \lim\limits_{\be\to\infty}\tau_{\al,\be}
%$$\tau_{\al,\infty}(\la)=\left[\frac{1}{\sqrt{\la}}\cdot\begin{pmatrix}1&(\la-1)\al\\ 0&\la\end{pmatrix}\right],$$
%where $\al$ resp.\ $\be$ is the real fixed points and $\la>0$.
%The parameter $\la>0$ is called the {\it shift parameter\/}.
The {\it axis\/}~$\ell(g)$ of a hyperbolic element $g$ is the geodesic between the fixed points of~$g$,
oriented from the repelling fixed point to the attracting fixed point.
%%from $\be$ to $\al$ if $\la>1$ and from $\al$ to $\be$ if $\la<1$.
The element $g$ preserves the geodesic $\ell(g)$.
%% and moves the points on this geodesic in the direction of the orientation.
We call a hyperbolic element with attracting fixed point~$\al$ and repelling fixed point~$\be$ {\it positive\/}
if~$\al<\be$.
The {\it shift parameter\/} of a hyperbolic element~$g$ is the minimal displacement~$\inf_{x\in\hyp}\,d(x,g(x))$.
%The map $\la\mapsto\tau_{\al,\be}(\la)$
%defines a homomorphism $\r_+\to G$ (with respect to the multiplicative structure on $\r_+$).
%We have
%$$(\tau_{\al,\be}(\la))^{-1}=\tau_{\al,\be}(\la^{-1})=\tau_{\be,\al}(\la).$$
An element is called {\it parabolic\/} if it has one fixed point, which is on the boundary~$\dd\hyp$.
%A parabolic element with real fixed point $\al$ is of the form
%$$\pi_{\al}(\la)=\left[\begin{pmatrix}1-\la\al&\la\al^2\\ -\la&1+\la\al\end{pmatrix}\right].$$
%A parabolic element with fixed point $\infty$ is of the form
%$$\pi_{\infty}(\la)=\left[\begin{pmatrix}1&\la\\ 0&1\end{pmatrix}\right].$$
We call a parabolic element $g$ with fixed point~$\al$ {\it positive\/} if $g(x)>x$ for all~$x\in\r\backslash\{\al\}$.
%The map $\la\mapsto\pi_{\al}(\la)$ defines a homomorphism $\r\to G$ (with respect to the additive structure on $\r$).
%We have
%$$(\pi_{\al}(\la))^{-1}=\pi_{\al}(-\la).$$
An element that is neither hyperbolic nor parabolic is called {\it elliptic\/}.
It has one fixed point that is in $\hyp$.
Given a base-point $x\in\hyp$ and a real number $\varphi$,
let $\rho_x(\varphi)\in G$ denote the rotation through angle $\varphi$ counter-clockwise about the point $x$.
Any elliptic element is of the form $\rho_x(\varphi)$, where $x$ is the fixed point.
Thus we obtain a $2\pi$-periodic homomorphism $\rho_x:\r\to G$ (with respect to the additive structure on $\r$).
%We have
%$$\rho_x(\varphi+2\pi)=\rho_x(\varphi)\quad\hbox{and}\quad(\rho_x(\varphi))^{-1}=\rho_x(-\varphi).$$
%For the fixed point $x=i\in\hyp$ we have
%$$\rho_i(\varphi)=\left[\begin{pmatrix} \cos\frac{\varphi}2&-\sin\frac{\varphi}2\\ \sin\frac{\varphi}2&\cos\frac{\varphi}2\end{pmatrix}\right].$$
%For a fixed point $x\in\hyp\bs\{i\}$ we obtain $\rho_x(\varphi)=\tau\circ\rho_i(\varphi)\circ\tau^{-1}$, where $\tau$ is the hyperbolic element in~$G$ such that $\tau(i)=x$.
%\subsection{Classification of elements in the covering groups}
Elements of $\tG$ resp.~$\Gm$ can be classified with respect to the fixed point behavior of action on $\hyp$ of their image in~$G$.
We say that an element of~$\tG$ resp.~$\Gm$ is {\it hyperbolic\/}, {\it parabolic\/} resp.\ {\it elliptic\/} if its image in~$G$ has this property.

%\subsection{One-parameter-subgroups of $G$ and $\tG$}

%\myskip
%The homomorphisms
%$$\tau_{\al,\be}:\r_+\to G,\quad \pi_{\al}:\r\to G,\quad\text{resp.}\quad \rho_x:\r\to G$$
%define one-parameter-subgroups in the group~$G=\PSL$.

%\myskip
%Each of the homomorphisms
%$\tau_{\al,\be}:\r_+\to G$, $\pi_{\al}:\r\to G$, resp.\ $\rho_x:\r\to G$
%lifts to a unique homomorphism
%$$t_{\al,\be}:\r_+\to\tG,\quad p_{\al}:\r\to\tG\quad\text{resp.}\quad r_x:\r\to\tG$$
%into the universal covering group.
%The elements $t_{\al,\be}(\la)$, $p_{\al}(\la)$, resp.\ $r_x(\xi)$
%are hyperbolic, parabolic, resp.\ elliptic.

%\myskip
%For hyperbolic elements we obtain for all $\la\in\r_+$ that
%$$\varphi(t_{\al,\beta}(\la))\in(-\pi,\pi)$$
%because of $\mu(\tau_{\al,\beta}(1))=1$ and $\mu(\tau_{\al,\beta}(\la))\ne-1$.

%\myskip
%Similarly, for parabolic elements we obtain for all $\la\in\r$ that
%$$\varphi(p_{\al}(\la))\in(-\pi,\pi)$$
%because of $\mu(\pi_{\al}(0))=1$ and $\mu(\pi_{\al}(\la))\ne-1$.

%\myskip
%Now let us consider elliptic elements.
%For the trace of the rotation $\rho_x(\xi)$ we have
%$$|\trace(\rho_x(\xi))|=|\trace(\rho_i(\xi))|=|2\cos(\xi/2)|,$$
%hence $\mu(\rho_x(\xi))=-1$ if and only if $\xi\in\pi+2\pi\cdot\z$.
%Because of $\mu(\rho_x(0))=1$ this implies for $\xi\in(-\pi+2\pi k,\pi+2\pi k)$ that
%$$\varphi(r_x(\xi))\in(-\pi+2\pi k,\pi+2\pi k).$$
%We obtain $\lev(r_x(\xi))=k$ for $\xi\in(-\pi+2\pi k,\pi+2\pi k]$.

\subsection{Central elements in covering groups of $G={\rm PSL}(2,\mathbb R)$}

\myskip
%The homomorphisms $\rho_x:\r\to G$ define one-parameter-subgroups in the group~$G=\PSL$.
The homomorphism $\rho_x:\r\to G$ lifts to a unique homomorphism $r_x:\r\to\tG$ into the universal covering group.
Since $\rho_x(2\pi\ell)=\id$ for~$\ell\in\z$,
it follows that the lifted element $r_x(2\pi\ell)$ belongs to the centre $Z(\tG)$ of~$\tG$.
Note that the element $r_x(2\pi\ell)$ depends continuously on $x$.
But the centre of~$\tG$ is discrete, so this element must remain constant,
thus~$r_x(2\pi\ell)$ does not depend on~$x$.
The centre~$Z(\tG)$ of~$\tG$ is equal to the pre-image of the identity element under the projection~$\tG\to G$, hence
$Z(\tG)=\{r_x(2\pi\ell)\st\ell\in\z\}$.
Let $\centrgen=r_x(2\pi)$ for some (and hence for any)~$x$ in~$\hyp$.
The element~$\centrgen$ is one of the two generators of the centre of~$\tG$
since any other element of the centre is of the form~$r_x(2\pi\ell)=(r_x(2\pi))^{\ell}=\centrgen^{\ell}$.
We would also like to point out that for the lift of an elliptic element $\rho_x(2\pi/p)$ of finite order~$p$
we have $(r_x(2\pi/p))^p=r_x(2\pi)=\centrgen$.

\subsection{Definition of a level function}

%\subsection{Decomposition of the subset of hyperbolic and parabolic elements in $\tG$ and $\Gm$ into sheets}
%\label{subsec-sheets}

\myskip
%Let $\Xi$ be the subset of~$G=\PSL$ that consists of all hyperbolic and parabolic elements of~$G$
%(including the identity element).
Let $\Delta$ be the set of all elliptic elements of order~$2$ in~$G$.
Let $\Xi$ be the complement in~$G$ of the set~$\De$.
The space~$G$ is homeomorphic to the open solid torus~$\s^1\times\c$.
In~\cite{JN} Jankins and Neumann give an explicit homeomorphism (see~\cite{JN}, Apendix).
The image of the set~$\De$ under this homeomorphism is~$\{*\}\times\c$.
%and describe the image of the subset~$\Xi$ under this homeomorphism (see~\cite{JN}, \S~1, Figure~1).
From this description it follows in particular that the subset~$\Xi$ is simply connected.
The pre-image~$\tilde\Xi$ of the subset~$\Xi$ in~$\tG$ consists of infinitely many connected components.
Each connected component of the subset~$\tilde\Xi$ contains one and only one pre-image of the identity element of~$G$,
\ie one and only one element of the centre of~$\tG$.
%The elements of the one-parameter-subgroups~$t_{\al,\be}(\la)$ resp.~$p_{\al}(\la)$
%are contained in the same connected component of the subset~$\tilde\Xi$
%as the identity element~$\tilde e$ of the group~$\tG$.

\begin{mydef}
If an element of~$\tG$ is contained in the same connected component of the set~$\tilde\Xi$ as the central element~$\centrgen^k$, $k\in\z$,
we say that the element is {\it at the level\/}~$k$ and set the {\it level function\/}~$\lev$ on this element to be equal to~$k$.
%Any hyperbolic resp.\ parabolic resp.\ elliptic element in~$\tG$ is of the form
%$t_{\al,\be}(\la)\cdot\centrgen^k$ resp.\ $p_{\al}(\la)\cdot\centrgen^k$ resp.\ $r_x(\xi)$.
%For elements written in this form we have
%$$\lev(t_{\al,\be}(\la)\cdot\centrgen^k)=k,\quad\lev(p_{\al}(\la)\cdot\centrgen^k)=k.$$
%and $\lev(r_x(\xi))=k$ if and only if $\xi\in(-\pi+2\pi k,\pi+2\pi k]$.
For pre-images of elliptic elements of order~$2$ we set $\lev(r_x(\xi))=k$ for $\xi=\pi+2\pi k$.
\end{mydef}

\begin{rem}
For elliptic elements we have $\lev(r_x(\xi))=k\iff\xi\in(-\pi+2\pi k,\pi+2\pi k]$.
\end{rem}

\begin{mydef}
We define the {\it level function\/}~$\levm$ on elements of~$\Gm=\tG/(m\cdot Z(\tG))$ by
$\levm(g~\mod~(m\cdot Z(\tG)))=\lev(g)~\mod~m\quad\text{for}\quad g\in\tG$.
(All equations involving~$\levm$ are to be understood as equations in~$\z/m\z$, \ie equations modulo~$m$.)
\end{mydef}

\begin{mydef}
The canonical lift of an element~$\bar C$ in~$G$ into~$\tG$ is an element~$\tilde C$ in~$\tG$
such that~$\pi(\tilde C)=\bar C$ and~$\lev(\tilde C)=0$.
The canonical lift of an element~$\bar C$ in~$G$ into~$\Gm$ is an element~$\tilde C$ in~$\Gm$
such that~$\pi(\tilde C)=\bar C$ and~$\levm(\tilde C)=0$.
\end{mydef}

\subsection{Properties of the level function}

\label{mult-tG}

In this subsection we study the behavior of the level function~$\levm$ under inversion (Lemma~\ref{lem-inv}),
conjugation (Lemma~\ref{lem-conj}) and multiplication in some special cases (Lemma~\ref{lem-product-of-crossing-hyps}).
We shall obtain further statements about the behavior of the level function under multiplication in Corollary~\ref{cor-product-dont-intersect}.

In this section we shall repeatedly use the following fact:
Connected components of the set~$\tilde\Xi$ are separated from each other by connected components of the set~$\tilde\De$ of all pre-images of (elliptic) elements of order~$2$.
If a path~$\ga$ in~$\tG$ avoids all pre-images of elements of order~$2$, i.e.\ avoids~$\tilde\De$,
then it means that the path~$\ga$ remains in the same component of the set~$\tilde\Xi$ and therefore the level function~$\lev$ is constant along~$\ga$.

%The main results of this subsection (Lemma~\ref{lem-product-of-hyps}, \ref{lem-product-of-hyp-and-par}, and \ref{lem-product-of-pars})
%are statements about the behavior of $\levm$ under multiplication.

%\myskip
%In this subsection let us denote by $[\cdot]$ the image of an element in~$\tG$ under the covering map $\tG\to\Gm$.

\begin{lem}
\label{lem-inv}
The equation $\lev(A^{-1})=-\lev(A)$ is satisfied for any element~$A$ in~$\tG$ with exception of pre-images of elliptic elements of order~$2$.
The equation $\levm(A^{-1})=-\levm(A)$ is satisfied for any element~$A$ in~$\Gm$ with exception of pre-images of elliptic elements of order~$2$.
\end{lem}

\begin{proof}
We shall prove the statement about the level function~$\lev$ on~$\tG$, the statement about the level function~$\levm$ on~$\Gm$ follows immediately.
Let $A\in G$ and let $k=\lev(A)$, then $A$ is in the same connected component of~$\tilde\Xi$ as~$u^k$.
Let $\ga$ be the path in~$\tilde\Xi$ that connects~$A$ with~$u^k$.
Let the path~$\de$ be given by~$\de(t)=(\ga(t))^{-1}$.
The path~$\de$ connects~$A^{-1}$ with~$u^{-k}$.
The path~$\ga$ remains in the same component of~$\tilde\Xi$, i.e.\ it avoids pre-images of elliptic elements of order~$2$.
Consequently, the path~$\de$ also avoids pre-images of elements of order~$2$,
i.e.\ it remains in the same component of~$\tilde\Xi$.
Thus the element~$A^{-1}$ is in the same connected component of~$\tilde\Xi$ as~$u^{-k}$,
i.e.\ $\lev(A^{-1})=-k=-\lev(A)$.
%The hyperbolic resp.\ parabolic element $A$
%is of the form $[t_{\al,\be}(\la)\cdot\centrgen^k]$ resp.\ $[p_{\al}(\la)\cdot\centrgen^k]$.
%If $A=[t_{\al,\be}(\la)\cdot \centrgen^k]$
%then $A^{-1}=[t_{\al,\be}(\la^{-1})\cdot \centrgen^{-k}]$ and $\levm(A^{-1})=-k=-\levm(A)$.
%If $A=[p_{\al}(\la)\cdot\centrgen^k]$
%then $A^{-1}=[p_{\al}(-\la)\cdot\centrgen^{-k}]$ and $\levm(A^{-1})=-k=-\levm(A)$.
\end{proof}

\begin{lem}
\label{lem-conj}
For any elements~$A$ and $B$ in~$\tG$ we have $\lev(BAB^{-1})=\lev(A)$.
For any elements~$A$ and $B$ in~$\Gm$ we have $\levm(BAB^{-1})=\levm(A)$.
\end{lem}

\begin{proof}
We shall prove the statement about the level function~$\lev$ on~$\tG$, the statement about the level function~$\levm$ on~$\Gm$ follows immediately.
The element~$B$ can be connected to the unit element in $\tG$ via a path $\be:I\to\tG$, where $I$ is some closed interval.
The path $\ga:I\to\tG$ given by $\ga(t)=\be(t)\cdot A\cdot(\be(t))^{-1}$ connects the elements~$A$ and~$B\cdot A\cdot B^{-1}$.
%We have $|\trace(\ga(t))|=|\trace(\be(t)\cdot A\cdot(\be(t))^{-1})|=|\trace(A)|\ne0$ for all $t\in I$.
%(Here $|\trace X|$ for an element $X$ in $\Gm$ is defined as $|\trace\bX|$, where $\bX$ is the projection of $X$ in $G$.)
If $A$ is not in~$\tilde\De$ then the same is true for the conjugate~$\ga(t)$ of~$A$,
hence the path~$\ga$ remains in the same component of the set~$\tilde\Xi$.
Thus $\lev$ is constant along~$\ga$, in particular $\lev(B\cdot A\cdot B^{-1})=\lev(A)$.
If $A$ is in~$\tilde\De$ then the conjugate~$\ga(t)$ of~$A$ is also in~$\tilde\De$,
hence the path~$\ga$ remains in the same component of the set~$\tilde\De$.
Thus $\lev$ is constant along~$\ga$, in particular $\lev(B\cdot A\cdot B^{-1})=\lev(A)$.
\end{proof}

\begin{lem}
\label{lem-product-of-crossing-hyps}
If the axes of two hyperbolic elements $A$ and~$B$ in~$\tG$ intersect then $\lev(AB)=\lev(A)+\lev(B)$.
If the axes of two hyperbolic elements $A$ and~$B$ in~$\Gm$ intersect then $\levm(AB)=\levm(A)+\levm(B)$.
\end{lem}

\begin{proof}
Let $\ell_A$ resp.~$\ell_B$ be the axes of~$A$ resp.~$B$.
Let $x$ be the intersection point of~$\ell_A$ and~$\ell_B$.
Any hyperbolic transformation with the axis~$\ell_A$ is a product of a rotation by~$\pi$ at some point~$y\ne x$ on~$\ell_A$ and a rotation by~$\pi$ at the point~$x$.
Similarly any hyperbolic transformation with the axis~$\ell_B$ is a product of a rotation by~$\pi$ at the point~$x$ and a rotation by~$\pi$ at some point~$z\ne x$ on~$\ell_B$.
Hence the product of any hyperbolic transformation with the axis~$\ell_A$ and any hyperbolic transformation with the axis~$\ell_B$
is a product of a rotation by~$\pi$ at a point~$y\ne x$ on~$\ell_A$ and a rotation by~$\pi$ at a point~$z\ne x$ on~$\ell_B$,
\ie it is a hyperbolic transformation with an axis going through the points~$y$ and~$z$.
Thus the product of two hyperbolic elements with distinct but intersecting axes is always a hyperbolic element.

\myskip
We shall prove the statement about the level function~$\lev$ on~$\tG$, the statement about the level function~$\levm$ on~$\Gm$ follows immediately.
Assume without loss of generality that the elements~$A,B\in\tG$ satisfy the conditions~$\lev(A)=\lev(B)=0$.
We want to show that~$\lev(AB)=0$.
Let us deform the elements~$A$ and~$B$.
We shall not change the axes but decrease the shift parameters, then the product tends to the identity element.
On the other hand we have explained that the product remains hyperbolic, i.e.\ stays in~$\tilde\Xi$.
Therefore the value of~$\lev$ on the product remains constant, i.e.\ $\lev(AB)=\lev(\id)=0$.
\end{proof}

\section{Level functions on lifts of Fuchsian groups}

\label{sec-levels-on-lifts}

\subsection{Lifting elliptic cyclic subgroups}

\begin{lem}

\label{lem-lifting-elliptics}

Let $\lat$ be an elliptic cyclic Fuchsian group of order~$p$.

\begin{enumerate}[1)]
\item
Let us assume that the numbers~$p$ and~$m$ are relatively prime.
Then the lift~$\lats$ of~$\lat$ into~$\Gm$ exists and is unique.
There is a unique element~$n\in\z/m\z$ such that~$p\cdot n+1=0\mod~m$.
The lift~$\lats$ is then determined by the following property:
If the elliptic element~$\ga=\rho_x(2\pi/p)$ is a generator of the group~$\lat$,
then the lift~$\lats$ is generated by the pre-image~$\tilde\ga$ of~$\ga$ in~$\Gm$ such that~$\levm(\tilde\ga)=n$.
\item
Let us assume that the numbers~$a$ and~$m$ are not relatively prime.
Then the group~$\lat$ can not be lifted into~$\Gm$.
\end{enumerate}
\end{lem}

\begin{proof}
Let $\ga=\rho_x(2\pi/p)$ be a generator of the group $\lat$.
To lift~$\lat$ into~$\Gm$ we have to find an element~$\tilde\ga$ in the pre-image of~$\ga$ in~$\Gm$
such that~$\tilde\ga^p=1$.
The pre-image of~$\ga$ in~$\Gm$ can be described as the coset $\{\centrgen^n\cdot r_x(2\pi/p)\st n\in\z/m\z\}$.
For the element~$r_x(2\pi/p)$ we obtain $(r_x(2\pi/p))^p=r_x(2\pi)=\centrgen$.
Hence for an element~$\centrgen^n\cdot r_x(2\pi/p)$ we obtain
$$
  (\centrgen^n\cdot r_x(2\pi/p))^p
  =\centrgen^{np}(r_x(2\pi/p))^p
  =\centrgen^{np+1}.
$$
Therefore $(\centrgen^n\cdot r_x(2\pi/p))^p=1$ if and only if $n\cdot p+1=0\mod~m$.
There exists~$n\in\z/m\z$ with~$n\cdot p+1=0\mod~m$ if and only if the numbers~$p$ and~$m$ are relatively prime.
Hence for not relatively prime $p$ and~$m$ it is impossible to lift $\lat$ into~$\Gm$.
For relatively prime~$p$ and~$m$ there is a unique lift of~$\lat$ into~$\Gm$
generated by $\centrgen^n\cdot r_x(2\pi/p)$ with~$\levm(\centrgen^n\cdot r_x(2\pi/p))=n$ and~$n\cdot p+1=0\mod~m$.
\end{proof}

\subsection{Finitely generated Fuchsian groups}

\myskip
In this section we are going to describe finitely generated (co-compact) Fuchsian groups using standard sets of generators.
The following definitions follow~\cite{Zieschang:book}:

% \cite{Zieschang:book}, p.~141
\begin{mydef}
A {\it Riemann factor surface\/} or {\it Riemann orbifold\/}~$(P,Q)$ of signature
$$(g;l_h,l_p,l_e:p_1,\dots,p_{l_e})$$
is a topological surface~$P$ of genus~$g$ with $l_h$ holes and $l_p$ punctures
and a set~$Q=\{(x_1,p_1),\dots,(x_{l_e},p_{l_e})\}$ of points~$x_i$ in~$P$ equipped with orders~$p_i$
such that~$p_i\in\z$, $p_i\ge2$ and~$x_i\ne x_j$ for~$i\ne j$.
The set~$Q$ is called the {\it marking\/} of the Riemann factor surface~$(P,Q)$.
\end{mydef}

% \cite{Zieschang:book}, p.~142
\begin{mydef}
Let $(P,Q=\{(x_1,p_1),\dots,(x_{l_e},p_{l_e})\})$ be a Riemann factor surface.
Two curves~$\ga_0$ and~$\ga_1$ which do not pass through exceptional points~$x_i\in Q$ are called $Q$-{\it homotopic\/} if $\ga_0$ can be deformed into~$\ga_1$ by a finite sequence of the following processes:
\begin{enumerate}[(a)]
\item
Homotopic deformations with fixed starting point such that during the deformation no exceptional point is encountered.
\item
Omitting a subcurve of~$\ga_i$ which does not contain the starting point of~$\ga_i$ and is of the form~$\de^{\pm p_i}$,
where~$\de$ is a curve on~$P$ which bounds a disk that contains exactly one exceptional point~$x_i$ in the interior.
\item
Inserting into~$\ga_i$ a subcurve which does not contain the starting point of~$\ga_i$ and is of the form~$\de^{\pm p_i}$,
where~$\de$ is a curve on~$P$ which bounds a disk that contains exactly one exceptional point~$x_i$ in the interior.
\end{enumerate}
Two curves~$\ga_0$ and~$\ga_1$ which do not pass through exceptional points~$x_i\in Q$ are called {\it freely $Q$-homotopic\/} if the base point may be moved during the deformations.
\end{mydef}

% \cite{Zieschang:book}, p.~143
\begin{mydef}
Let $(P,Q=\{(x_1,p_1),\dots,(x_{l_e},p_{l_e})\})$ be a Riemann factor surface and~$p\in P\backslash Q$.
Then the set of $Q$-homotopy classes of curves starting and ending in~$p$ forms a group.
This group is called the $Q$-{\it fundamental group\/} or the {\it orbifold fundamental group\/}
and denoted by~$\pi^Q(P,p)$ or~$\piorb(P,p)$.
\end{mydef}

% \cite{Zieschang:book}, p.~143
\begin{mydef}
Let $\lat$ be a Fuchsian group.
The quotient~$P=\hyp/\lat$ is a surface and the projection~$\Psi:\hyp\to P$ is a branched cover.
Let $Q$ consist of the branching points and the corresponding orders.
Then~$(P,Q)$ is a factor surface.
We say that the factor surface~$(P,Q)$ is defined by~$\lat$.
\end{mydef}

% \cite{Zieschang:book}, p.~143
\begin{prop}
Let $\lat$ be a Fuchsian group, $(P,Q)$ the corresponding factor surface and~$p\in P\backslash Q$.
Then~$\piorb(P,p)\cong\lat$.
\end{prop}

% \cite{Zieschang:book}, p.~150
\begin{mydef}
\label{def-standard-basis}
%We define the product $ab$ of two contours $a$ and~$b$ in $\piorb(P,p)$ as the contour given by the path of~$b$ followed by the path of~$a$.
A {\it canonical system of curves\/} on a Riemann factor surface~$(P,Q=\{(x_1,p_1),\dots,(x_{l_e},p_{l_e})\})$
of signature $(g;l_h,l_p,l_e:p_1,\dots,p_{l_e})$ is a set of simply closed curves based at a point~$p\in P$
%$$\{a_1,b_1,\dots,a_g,b_g,c_{g+1},\dots,c_n\}$$
%with a single defining relation
%$$\prod_{i=1}^g [a_i,b_i] \prod_{i=g+1}^n c_i=1$$
%and represented by a set of simple contours
$$\{\tilde a_1,\tilde b_1,\dots,\tilde a_g,\tilde b_g,\tilde c_{g+1},\dots,\tilde c_{g+l_h+l_p+l_e}\},$$
where $n=g+l_h+l_p+l_e$, with the following properties:
\begin{enumerate}[1)]
\item
The contour $\tilde c_i$ encloses a hole in $P$ for $i=g+1,\dots,g+l_h$,
a puncture for $i=g+l_h+1,\dots,g+l_h+l_p$
and the marking point~$x_i$ for $i=g+l_h+l_p+1,\dots,n$.
\item
Any two curves only intersect at the point~$p$.
%$\tilde a_i\cap\tilde b_j=\tilde a_i\cap\tilde c_j=\tilde b_i\cap\tilde c_j=\tilde c_i\cap\tilde c_j=\{p\}$.
\item
In a neighbourhood of the point $p$, the curves are placed as is shown in Figure~\ref{fig-basis}.
\item
The system of curves cuts the surface~$P$ into~$l_h+l_p+l_e+1$~connected components of which $l_p+l_e$ are homeomorphic to a disc with a hole,
$l_h+1$ are discs.
The last disc has boundary
$$
  \tilde a_1\tilde b_1\tilde a_1^{-1}\tilde b_1^{-1}\dots\tilde a_g\tilde b_g\tilde a_g^{-1}\tilde b_g^{-1}
  \tilde c_g\dots\tilde c_n.
$$
\end{enumerate}
If $\{\tilde a_1,\tilde b_1,\dots,\tilde a_g,\tilde b_g,\tilde c_{g+1},\dots,\tilde c_{g+l_h+l_p+l_e}\}$ is a canonical system of curves,
then we call the corresponding set $\{a_1,b_1,\dots,a_g,b_g,c_{g+1},\dots,c_{g+l_h+l_p+l_e}\}$ of elements in the orbifold fundamental group~$\piorb(P,p)$
a {\it standard basis\/} or a {\it standard set of generators\/} of~$\piorb(P,p)$.
\end{mydef}

% Picture

\begin{figure}
  \begin{center}
    \forcehmode
      \bgroup
        \beginpicture
          \setcoordinatesystem units <25 bp,25 bp>
          \multiput {\phantom{$\bullet$}} at -2 -2 2 2 /
          \plot 0 0 -2 0.5 /
          \arrow <7pt> [0.2,0.5] from -2 0.5 to -1 0.25
          \plot 0 0 -2 1 /
          \arrow <7pt> [0.2,0.5] from 0 0 to -1 0.5
          \plot 0 0 -1.6 1.6 /
          \arrow <7pt> [0.2,0.5] from 0 0 to -0.8 0.8
          \plot 0 0 -1 2 /
          \arrow <7pt> [0.2,0.5] from -1 2 to -0.5 1
          \plot 0 0 1 2 /
          \arrow <7pt> [0.2,0.5] from 1 2 to 0.5 1
          \plot 0 0 1.6 1.6 /
          \arrow <7pt> [0.2,0.5] from 0 0 to 0.8 0.8
          \plot 0 0 2 1 /
          \arrow <7pt> [0.2,0.5] from 0 0 to 1 0.5
          \plot 0 0 2 0.5 /
          \arrow <7pt> [0.2,0.5] from 2 0.5 to 1 0.25
          \plot 0 0 1.6 -1.6 /
          \arrow <7pt> [0.2,0.5] from 1.6 -1.6 to 0.8 -0.8
          \plot 0 0 1 -2 /
          \arrow <7pt> [0.2,0.5] from 0 0 to 0.5 -1
          \plot 0 0 -1 -2 /
          \arrow <7pt> [0.2,0.5] from -1 -2 to -0.5 -1
          \plot 0 0 -1.6 -1.6 /
          \arrow <7pt> [0.2,0.5] from 0 0 to -0.8 -0.8
          \put {$a_1$} [r] <-2pt,0pt> at -2 0.5
          \put {$b_1$} [r] <-2pt,0pt> at -2 1
          \put {$a_1$} [br] <0pt,2pt> at -1.6 1.6
          \put {$b_1$} [br] <0pt,2pt> at -1 2
          \put {$a_g$} [bl] <0pt,2pt> at 1 2
          \put {$b_g$} [bl] <0pt,2pt> at 1.6 1.6
          \put {$a_g$} [l] <2pt,0pt> at 2 1
          \put {$b_g$} [l] <2pt,0pt> at 2 0.5
          \put {$c_{g+1}$} [tl] <0pt,-2pt> at 1 -2
          \put {$c_{g+1}$} [tl] <0pt,-2pt> at 1.6 -1.6
          \put {$c_n$} [tr] <0pt,-2pt> at -1.6 -1.6
          \put {$c_n$} [tr] <0pt,-2pt> at -1 -2
          \put {$\dots$} at 0 1
          \put {$\dots$} at 0 -1
        \endpicture
      \egroup
  \end{center}
  \caption{Canonical system of curves}
  \label{fig-basis}
\end{figure}

\begin{mydef}
\label{def-short-seq-set}
A {\it sequential set of signature\/}~$(0;l_h,l_p,l_e:p_1,\dots,p_{l_e})$ with $l_h+l_p+l_e=3$
is a triple of elements $(C_1,C_2,C_3)$ in~$G$
such that the element $C_i$ is hyperbolic for~$i=1,\dots,l_h$, parabolic for~$i=l_h+1,\dots,l_h+l_p$
and elliptic of order~$p_{i-l_h-l_p}$ for~$i=l_h+l_p+1,\dots,l_h+l_p+l_e=3$,
their product is $C_1\cdot C_2\cdot C_3=1$,
and for some element $A\in G$ the elements~$\{\tilde C_i=A C_i A^{-1}\}_{i=1,2,3}$ are positive, have finite fixed points
and satisfy~$\tilde C_1<\tilde C_2<\tilde C_3$.
(Figure~\ref{fig-axes-seqset} illustrates the position of the axes of the elements~$\tilde C_i$
for a sequential set of signature $(0,3,0)$, \ie when all elements are hyperbolic.)
%It is clear what the similar picture looks like in presence of parabolic elements.)
\end{mydef}

% Picture

\begin{figure}
  \begin{center}
    \forcehmode
      \bgroup
        \beginpicture
          \setcoordinatesystem units <25 bp,25 bp>
          \multiput {\phantom{$\bullet$}} at -5 -1 5 2 /
          \circulararc 180 degrees from -2 0 center at -3 0
          \arrow <7pt> [0.2,0.5] from -3.99 0.04 to -4 0
          \put {$\ell(AC_1A^{-1})$} [b] <0pt,\baselineskip> at -3 1
          \circulararc 180 degrees from  1 0 center at  0 0
          \arrow <7pt> [0.2,0.5] from -0.99 0.04 to -1 0
          \put {$\ell(AC_2A^{-1})$} [b] <0pt,\baselineskip> at 0 1
          \circulararc 180 degrees from  4 0 center at  3 0
          \arrow <7pt> [0.2,0.5] from  2.01 0.04 to  2 0
          \put {$\ell(AC_3A^{-1})$} [b] <0pt,\baselineskip> at 3 1
          \plot -5 0 5 0 /
        \endpicture
      \egroup
  \end{center}
  \caption{Axes of a sequential set of signature $(0;3,0,0)$}
  \label{fig-axes-seqset}
\end{figure}

\begin{mydef}
\label{def-long-seq-set}
A {\it sequential set of signature\/}~$(0;l_h,l_p,l_e:p_1,\dots,p_{l_e})$
is a tuple of elements $(C_1,\dots,C_{l_h+l_p+l_e})$ in~$G$
%with $n=l_h+l_p+l_e$ in~$G$
such that the element $C_i$ is hyperbolic for~$i=1,\dots,l_h$,
parabolic for~$i=l_h+1,\dots,l_h+l_p$ and elliptic of order~$p_{i-l_h-l_p}$ for~$i=l_h+l_p+1,\dots,l_h+l_p+l_e$,
and for any $i\in\{2,\dots,n-1\}$ the triple $(C_1\cdots C_{i-1},C_i,C_{i+1}\cdots C_n)$ is a sequential set.
%(of signature $(0;3,0,0)$, $(0;2,1,0)$, $(0;2,0,1)$, $(0;1,2,0)$, $(0;1,0,2)$, $(0;1,1,1)$, $(0;0,3,0)$, $(0;0,2,1)$, $(0;0,1,2)$ or~$(0;0,0,3)$).
\end{mydef}

\begin{mydef}
\label{def-genus-seq-set}
A {\it sequential set of signature\/}~$(g;l_h,l_p,l_e:p_1,\dots,p_{l_e})$ is a tuple of elements
$$(A_1,\dots,A_g,B_1,\dots,B_g,C_{g+1},\dots,C_{g+l_h+l_p+l_e})$$
%with $n=l_h+l_p+l_e$
in~$G$ such that the elements $A_1,\dots,A_g,B_1,\dots,B_g$ are hyperbolic,
the element~$C_{g+i}$ is hyperbolic for~$i=1,\dots,l_h$, parabolic for~$i=l_h+1,\dots,l_h+l_p$
and elliptic of order~$p_{i-l_h-l_p}$ for~$i=l_h+l_p+1,\dots,l_h+l_p+l_e$,
and the tuple
$$(A_1,B_1A_1^{-1}B_1^{-1},\dots,A_g,B_gA_g^{-1}B_g^{-1},C_{g+1},\dots,C_{g+l_h+l_p+l_e})$$
is a sequential set of signature $(0;2g+l_h,l_p,l_e:p_1,\dots,p_{l_e})$.
\end{mydef}

The relation between sequential sets, standard bases, canonical systems of curves and Fuchsian groups was studied in~\cite{N1972}.
Details for the case of Fuchsian groups with elliptic elements can be found in~\cite{Zieschang:book}.
We recall here the results:

\begin{thm}
\label{seqset-fuchgr}
Let $V$~be a sequential set of signature $(g;l_h,l_p,l_e:p_1,\dots,p_{l_e})$.
For~$i=1,\dots,l_e$ let~$y_i\in\hyp$ be the fixed point of the corresponding elliptic element of order~$p_i$ in~$V$.
Let $P=\hyp/\lat$ and let $\Psi:\hyp\to P$ be the natural projection.
Let~$Q=\{(\Psi(y_1),p_1),\dots,(\Psi(y_{l_e}),p_{l_e})\}$.
Then the sequential set~$V$ generates a Fuchsian group~$\lat$
such that the Riemann factor surface $(P=\hyp/\lat,Q)$ is of signature $(g;l_h,l_p,l_e:p_1,\dots,p_{l_e})$.
The natural projection $\Psi:\hyp\to P$ maps the sequential set $V$ to a canonical system of curves on the factor surface~$(P,Q)$.
\end{thm}

\begin{thm}
\label{thm2.1inN}
Let $\lat$ be a Fuchsian group such that the factor surface $P=\hyp/\lat$ is of signature $(g;l_h,l_p,l_e:p_1,\dots,p_{l_e})$.
Let $p$ be a point in~$P$ which does not belong to the marking.
Let $\Psi:\hyp\to P$ be the natural projection.
Choose $q\in\Psi^{-1}(p)$ and let $\Phi:\lat\to\piorb(P,p)$ be the induced isomorphism.
Let
$$v=\{\tilde a_1,\tilde b_1,\dots,\tilde a_g,\tilde b_g,\tilde c_{g+1},\dots,\tilde c_n\}$$
be a canonical system of curves on~$P$.
In this case,
\begin{align*}
  V=\Phi^{-1}(v)&=\{\Phi^{-1}(a_1),\Phi^{-1}(b_1),\dots,\Phi^{-1}(a_g),\Phi^{-1}(b_g),\Phi^{-1}(c_{g+1}),\dots,\Phi^{-1}(c_n)\}\\
                &=\{A_1,B_1,\dots,A_g,B_g,C_{g+1},\dots,C_n\}
\end{align*}
is a sequential set of signature $(g;l_h,l_p,l_e:p_1,\dots,p_{l_e})$.
\end{thm}

\subsection{Lifting Fuchsian groups of genus $0$}

\begin{lem}
\label{lem-canon-lift-polygon-group}
Let $(0;l_h,l_p,l_e:p_1,\dots,p_{l_e})$ with~$l_h+l_p+l_e=n$ be the signature of the sequential set $(\bC_1,\dots,\bC_n)$.
For~$i=1,\dots,n$ let~$\tC_i$ be the canonical lift of~$\bC_i$ into~$\tG$ or~$\Gm$.
Let~$\centrgen$ be the generator of the centre~$Z(\tG)$ resp.~$Z(\Gm)$.
The element~$\centrgen$ is given by the element~$r_x(\pi)$ resp.\ its projection into~$\Gm$.
Then the elements~$\tC_1,\dots,\tC_n$ satisfy the following relations: $\tC_{l_h+l_p+i}^{p_i}=\centrgen$ for~$i=1,\dots,l_e$ and
$$\tC_1\cdot\cdots\tC_n=\centrgen^{n-2}.$$
\end{lem}

\begin{proof}
Let $\Pi$~be the canonical fundamental polygon for the group generated by the elements $\bC_1,\dots,\bC_n$ such that the generators~$\bC_i$ can be described as products
$$\bC_i=\si_i\si_{i+1}$$
of reflexions~$\si_1,\dots,\si_n$ in the edges of the polygon~$\Pi$ (suitably numbered).
Then $\si_i^2=\id$, and therefore
$$\bC_1\cdots\bC_n=(\si_1\si_2)(\si_2\si_3)\cdots(\si_{n-1}\si_n)(\si_n\si_1)=\id.$$
Lifting the elements~$\bC_i$ to their canonical lifts~$\tC_i$ in~$\tG$,
it follows that the product~$\tC_1\cdots\tC_n$ belongs to the centre~$Z(\tG)$.
As we vary the polygon~$\Pi$ continuously, this central element must also vary continuously.
But~$Z(\tG)$ is a discrete group, so $\tC_1\cdots\tC_n$ must remain constant.
In particular we can shrink the polygon~$\Pi$ down towards a point~$x$.
In the course of this continuous deformation of the fundamental polygon~$\Pi$ the hyperbolic and parabolic elements of the sequential set will become elliptic.
As we continue shrinking the polygon~$\Pi$ towards the point~$x$, the angles of the polygon tend to the angles~$\be_1,\dots,\be_n$ of some Euclidean $n$-sided polygon.
Thus the element~$\tC_i\in\tG$ tends towards the limit~$r_x(2\be_i)$, while the product~$\tC_1\cdots\tC_n$ tends towards the product
$$r_x(2\be_1)\cdots r_x(2\be_n)=r_x(2\be_1+\cdots+2\be_n).$$
Therefore, using the formula
$$\be_1+\cdots+\be_n=(n-2)\pi$$
for the sum of the angles of a Euclidean $n$-sided polygon, we see that the constant product~$\tC_1\cdots\tC_n$ must be equal to
$$r_x(2(n-2)\pi)=\centrgen^{n-2}.$$
Projecting into~$\Gm$ we get the corresponding statement in~$\Gm$.
\end{proof}

\begin{lem}
\label{lem-lift-polygon-group}
Let $(C_1,\dots,C_n)$ be an $n$-tuple of elements in~$\Gm$
such that their images $(\bC_1,\dots,\bC_n)$ in~$G$ form a sequential set of signature $(0;l_h,l_p,l_e:p_1,\dots,p_{l_e})$ with~$l_h+l_p+l_e=n$.
Then $C_1\cdots C_n=e$ if and only if
$$\levm(C_1)+\cdots+\levm(C_n)=-(n-2).$$
%Moreover $$\levm(C_1\cdot C_2)=\levm(C_1)+\levm(C_2)+1.$$
\end{lem}

\begin{proof}
For~$i=1,\dots,n$ let~$\tC_i$ be the canonical lift of~$\bC_i$ into~$\Gm$.
The elements~$C_i$ can be written in the form
$$C_i=\tC_i\cdot u^{\levm(C_i)},$$
therefore
$$
  C_1\cdots C_n
  =(\tC_1\cdot u^{\levm(C_1)})\cdots(\tC_n\cdot u^{\levm(C_n)})
  =(\tC_1\cdots\tC_n)\cdot c^{\levm(C_1)+\cdots+\levm(C_n)}.
$$
According to Lemma~\ref{lem-canon-lift-polygon-group} the product of the elements~$\tC_i$ is
$$\tC_1\cdot\cdots\tC_n=\centrgen^{n-2},$$
hence
$$C_1\cdots C_n=\centrgen^{n-2+\levm(C_1)+\cdots+\levm(C_n)}.$$
Therefore the product~$C_1\cdots C_n$ is equal to~$e$ if and only if the exponent of~$\centrgen$ in the last equation is divisible by~$m$, i.e.~if
$$\levm(C_1)+\cdots+\levm(C_n)\equiv-(n-2)\mod~m.\qedhere$$
\end{proof}

\begin{cor}
\label{cor-product-dont-intersect}
Let $(C_1,C_2,C_3)$ be an triple of elements in~$\Gm$ with~$C_1C_2C_3=e$.
Let $\bC_i$ be the image of the element~$C_i$ in~$G$.
Let $(\bC_1,\bC_2,\bC_3)$ be a sequential set of signature $(0;l_h,l_p,l_e:p_1,\dots,p_{l_e})$ with~$l_h+l_p+l_e=3$.
Then
$$\levm(C_1\cdot C_2)=\levm(C_1)+\levm(C_2)+1$$
if the element~$C_3$ is not of order~$2$ and
$$\levm(C_1\cdot C_2)=-\levm(C_1)-\levm(C_2)-1$$
if the element~$C_3$ is of order~$2$.
\end{cor}

\begin{proof}
According to Lemma~\ref{lem-lift-polygon-group} the elements~$C_i$ satisfy
$$\levm(C_1)+\levm(C_2)+\levm(C_3)\equiv-1\mod~m.$$
On the other hand~$C_1C_2C_3=e$ implies~$C_1C_2=C_3^{-1}$, hence
$$\levm(C_1C_2)=\levm(C_3^{-1})=-\levm(C_3)=\levm(C_1)+\levm(C_2)+1$$
if the element~$C_3$ is not of order~$2$ and
$$\levm(C_1C_2)=\levm(C_3^{-1})=\levm(C_3)=-\levm(C_1)-\levm(C_2)-1$$
if the element~$C_3$ is of order~$2$.
\end{proof}

\subsection{Lifting sets of generators of Fuchsian groups}

\label{first-rem-FG-into-tG}

%\myskip
%In this subsection let us denote by $[\cdot]$ the image of an element in~$\tG$ under the covering map $\tG\to\Gm$.

\begin{lem}
\label{lem-liftonerel}
Let $\lat$ be a Fuchsian group of signature $(g;l_h,l_p,l_e:p_1,\dots,p_{l_e})$ generated by the sequential set
$\bV=\{\bA_1,\bB_1,\dots,\bA_g,\bB_g,\bC_{g+1},\dots,\bC_n\}$, where $n=g+l_h+l_p+l_e$.
Let $V=\{A_1,B_1,\dots,A_g,B_g,C_{g+1},\dots,C_n\}$
be a set of lifts of the elements of the sequential set $\bV$ into~$\Gm$,
\ie the image of $A_i$, $B_i$ resp.~$C_j$ in $G$ is $\bA_i$, $\bB_i$ resp.~$\bC_j$.
Then the subgroup $\lats$ of $\Gm$ generated by $V$ is a lift of $\lat$ into~$\Gm$ if and only if
$$
  [A_1,B_1]\cdots[A_g,B_g]\cdot C_{g+1}\cdots C_n=e,\quad
%  \prod\limits_{i=1}^g\,[A_i,B_i]\cdot\prod\limits_{j=g+1}^n\,C_j=e,\quad 
  C_{g+l_h+l_p+i}^{p_i}=e\quad\text{for}~i=1,\dots,l_e.
$$
\end{lem}

\begin{proof}
For any choice of the set of lifts $V$ the restriction of the covering map
$\Gm\to G$ to the group $\lats$ generated by $V$ is a homomorphism with image $\lat$.
If the conditions of the lemma hold true, then the group~$\lats$ satisfies the same relations as the group~$\lat$,
hence this homomorphism is injective.
%The only relations in the group~$\lat$ are
%$$
%  [\bA_1,\bB_1]\cdots[\bA_g,\bB_g]\cdot\bC_{g+1}\cdots\bC_n=e,\quad
%%  \prod\limits_{i=1}^g\,[\bA_i,\bB_i]\cdot\prod\limits_{j=g+1}^n\,\bC_j=e,\quad
%  \bC_{g+l_h+l_p+j}^{p_j}=e\quad\text{for}~j=1,\dots,l_e,
%$$
%hence the equality
%$$
%  [A_1,B_1]\cdots[A_g,B_g]\cdot C_{g+1}\cdots C_n=e,\quad
%%  \prod\limits_{i=1}^g\,[A_i,B_i]\cdot\prod\limits_{j=g+1}^n\,C_j=e,\quad 
%  C_{g+l_h+l_p+j}^{p_j}=e\quad\text{for}~j=1,\dots,l_e.
%$$
%ensures injectivity of this homomorphism.
\end{proof}

\begin{lem}
\label{lem-lift-glm}
Let
$$\{A_1,B_1,\dots,A_g,B_g,C_{g+1},\dots,C_n\}$$
be a tuple of elements in~$\Gm$
such that the images
$$\{\bA_1,\bB_1,\dots,\bA_g,\bB_g,\bC_{g+1},\dots,\bC_n\}$$
in $G$ form a sequential set of signature $(g;l_h,l_p,l_e:p_1,\dots,p_{l_e})$ with~$g+l_h+l_p+l_e=n$.
Then
$$
  \prod\limits_{i=1}^g\,[A_i,B_i]\cdot\prod\limits_{j=g+1}^n\,C_j=e\iff 
  \sum\limits_{j=g+1}^n\,\levm(C_j)\equiv(2-2g)-(n-g)\mod~m.
$$
(in the case $n=g$ this means $2-2g\equiv0\mod~m$)
and for any~$i=1,\dots,l_e$
$$C_{g+l_h+l_p+i}^{p_i}=e\iff p_i\cdot\levm(C_{g+l_h+l_p+i})+1\equiv0\mod\,m.$$
\end{lem}

\begin{proof}
The case~$g=0$ was discussed in Lemma~\ref{lem-lift-polygon-group}.
We shall now reduce the general case to the case~$g=0$.
By definition of sequential sets the set
$$(\bA_1,\bB_1\bA_1^{-1}\bB_1^{-1},\dots,\bA_g,\bB_g\bA_g^{-1}\bB_g^{-1},\bC_{g+1},\dots,\bC_{n})$$
is a sequential set of signature $(0;2g+l_h,l_p,l_e)$, hence
$$
  \prod\limits_{i=1}^g\,[A_i,B_i]\cdot\prod\limits_{i=g+1}^n C_i
  =\prod\limits_{i=1}^g\,(A_i\cdot B_i A_i^{-1}B_i^{-1})\cdot\prod\limits_{i=g+1}^n C_i
  =e
$$
if and only if
\begin{align*}
  \sum\limits_{i=1}^g\,(\levm(A_i)+\levm(B_i A_i^{-1}B_i^{-1}))+\sum\limits_{i=g+1}^n\levm(C_i)
  &\equiv-(2g+(n-g)-2)\\
  &\equiv(2-2g)-(n-g)\mod~m.
\end{align*}
Invariance of the level function~$\levm$ under conjugation (Lemma~\ref{lem-conj}) implies that
$$\levm(B_i A_i^{-1}B_i^{-1})=\levm(A_i^{-1}).$$
Since~$A_i$ is not an element of order~$2$, 
$$\levm(A_i^{-1})=-\levm(A_i),$$
and hence~$\levm(B_i A_i^{-1} B_i^{-1})=-\levm(A_i)$ and
$$\levm(A_i)+\levm(B_i A_i^{-1}B_i^{-1})=\levm(A_i)-\levm(A_i)=0.$$
The last statement of the lemma follows from Lemma~\ref{lem-lifting-elliptics}.
\end{proof}

\begin{prop}
\label{prop-lift-exists}
Let $\lat$ be a Fuchsian group of signature $(g:p_1,\dots,p_r)$.
Let $\bV=\{\bA_1,\bB_1,\dots,\bA_g,\bB_g,\bC_{g+1},\dots,\bC_n\}$ be a sequential set that generates~$\lat$.
Then there exist lifts of~$\lat$ into~$\Gm$ if and only if the signature $(g:p_1,\dots,p_r)$ satisfies the following liftability conditions:
$\gcd(p_i,m)=1$ for~$i=1,\dots,r$ and
$$(p_1\cdots p_r)\cdot\left(\sum\limits_{i=1}^r\,\frac1{p_i}-(2g-2)-r\right)\equiv0\mod m.$$
Moreover, if the liftability conditions are satisfied then any set of lifts~$\{A_i,B_i\}$ of~$\{\bA_i,\bB_i\}$ into~$\Gm$
can be extended in a unique way to a set $\{A_i,B_i,C_j\}$ of lifts of~$\{\bA_i,\bB_i,\bC_j\}$ that generates a lift of~$\lat$ into~$\Gm$,
hence there are $m^{2g}$ different lifts of~$\lat$ into~$\Gm$.
\end{prop}

\begin{proof}
Let us first assume that there exists a lift of~$\lat$ into~$\Gm$.
Let $\{A_i,B_i,C_j\}$ be a set of lifts of~$\bar V$ as in Lemmas~\ref{lem-liftonerel} and~\ref{lem-lift-glm}.
Let~$n_i=\levm(C_{g+i})$.
Then according to Lemma~\ref{lem-lift-glm} we have $p_i\cdot n_i+1\equiv0\mod\,m$ for~$i=1,\dots,r$ and
$$(2g-2)+r+\sum\limits_{i=1}^r\,n_i\equiv0\mod\,m.$$
The first set of congruences implies that $p_i$ is prime with~$m$ for~$i=1,\dots,r$.
The last congruence implies that
\begin{align*}
  0
  &\equiv(p_1\cdots p_r)\cdot\left((2g-2)+r+\sum\limits_{i=1}^r\,n_i\right)\\
  &\equiv(p_1\cdots p_r)\cdot((2g-2)+r)+\sum\limits_{i=1}^r\,\frac{p_1\cdots p_r}{p_i}\cdot(p_i\cdot n_i)\\
  &\equiv(p_1\cdots p_r)\cdot((2g-2)+r)+\sum\limits_{i=1}^r\,\frac{p_1\cdots p_r}{p_i}\cdot(-1)\\
  &\equiv(p_1\cdots p_r)\cdot\left((2g-2)+r-\sum\limits_{i=1}^r\,\frac1{p_i}\right).
\end{align*}
Now let us assume that the liftability conditions are satisfied.
We want to construct a lift of~$\lat$ into~$\Gm$.
Since $p_i$ is prime with~$m$, we can choose~$n_i\in\z/m\z$ such that $p_i\cdot n_i+1\equiv0\mod\,m$ for~$i=1,\dots,r$.
Then
\begin{align*}
  &(p_1\cdots p_r)\cdot\left((2g-2)+r+\sum\limits_{i=1}^r\,n_i\right)\\
  &\equiv(p_1\cdots p_r)\cdot((2g-2)+r)+\sum\limits_{i=1}^r\,\frac{p_1\cdots p_r}{p_i}\cdot(p_in_i)\\
  &\equiv(p_1\cdots p_r)\cdot((2g-2)+r)+\sum\limits_{i=1}^r\,\frac{p_1\cdots p_r}{p_i}\cdot(-1)\\
  &\equiv(p_1\cdots p_r)\cdot\left((2g-2)+r-\sum\limits_{i=1}^r\,\frac1{p_i}\right)
  \equiv0\mod\,m.
\end{align*}
Since $\gcd(p_i,m)=1$,
the equality $(p_1\cdots p_r)\cdot\left((2g-2)+r+\sum\limits_{i=1}^r\,n_i\right)\equiv0\mod\,m$
implies $(2g-2)+r+\sum\limits_{i=1}^r\,n_i\equiv0\mod\,m$, i.e. $\sum\limits_{i=1}^r\,n_i\equiv(2-2g)-r\mod\,m$.
Let $V=\{A_i,B_i,C_j\}$ be any set of lifts of~$\bar V$ such that~$\levm(C_{g+i})=n_i$ for~$i=1,\dots,r$.
We have $p_i\cdot n_i+1\equiv0\mod\,m$ for~$i=1,\dots,r$ and $\sum\limits_{i=1}^r\,n_i\equiv(2-2g)-r\mod\,m$,
hence according to Lemma~\ref{lem-lift-glm} the set~$V$ generates a lift of~$\lat$ into~$\Gm$.
Since Lemma~\ref{lem-lift-glm} does not impose any conditions on the values~$\levm(A_i)$ and~$\levm(B_i)$ for~$i=1,\dots,g$,
any of $m^{2g}$ choices of these $2g$ values leads to a different lift of~$\lat$ into~$\Gm$. 
\end{proof}

\section{Higher Arf functions}

\label{sec-m-arf}

In~\cite{NP1} we introduced the notion of a higher Arf function
and used it to study moduli spaces of higher spin bundles on Riemann surfaces.
In this section we will introduce higher Arf functions on orbifolds,
and study their connection with Gorenstein automorphy factors.
%describe the connection between Gorenstein automorphy factors and higher Arf functions.

\subsection{Definition of higher Arf functions on orbifolds}

\label{m-arf}

\myskip
In this subsection we will define higher Arf functions on orbifolds (compare with subsection~4.1 in~\cite{NP1}).

\myskip
Let $\lat$ be a Fuchsian group of signature~$(g;l_h,l_p,l_e:p_1,\dots,p_{l_e})$ and $P=\hyp/\lat$ the corresponding orbifold.
Let $p\in P$.
Let $\Psi:\hyp\to P$ be the natural projection.
Choose $q\in\Psi^{-1}(p)$ and let $\Phi:\lat\to\piorb(P,p)$ be the induced isomorphism.
Let $\lats$ be a lift of $\lat$ in $\Gm$.

\begin{mydef}
\label{def-hsi-to-lats}
Let us consider a function $\hsi_{\lats}:\piorb(P,p)\to\z/m\z$ such that the following diagram commutes
$$
  \begin{CD}
   \lat            @>{\cong}>> \lats         \\
   @V{\Phi}VV         @VV{\levm|_{\lats}}V \\
   \piorb(P,p)     @>{\hsi_{\lats}}>> \z/m\z          \\
  \end{CD}
$$
% As for the function~$\levm$, all equations involving~$\hsi_{\lats}$ are to be understood as equations in~$\z/m\z$.
\end{mydef}

\begin{lem}
\label{lem-hsi-rules}
Let $\al$, $\be$, and $\ga$ be simple contours in $P$
intersecting pairwise in exactly one point~$p$.
Let $a$, $b$, and $c$ be the corresponding elements of $\piorb(P,p)$.
We assume that $a$, $b$, and $c$ satisfy the relations $a,b,c\ne1$ and $abc=1$.
Let $\<\cdot,\cdot\>$ be the intersection form on $\piorb(P,p)$.
Then for $\hsi=\hsi_{\lats}$
\begin{enumerate}[1.]
\item
If the elements~$a$ and~$b$ can be represented by a pair of simple contours in $P$ intersecting in exactly one point~$p$ with $\<a,b\>\ne0$,
then~$\hsi(ab)=\hsi(a)+\hsi(b)$.
\item
If $ab$ is in~$\piorbO(P,p)$
and the elements~$a$ and~$b$ can be represented by a pair of simple contours in $P$ intersecting in exactly one point~$p$ with $\<a,b\>=0$
and placed in a neighbourhood of the point~$p$ as shown in Figure~\ref{fig-pos-pair},
then~$\hsi(ab)=\hsi(a)+\hsi(b)+1$ if the element~$ab$ is not of order~$2$
and~$\hsi(ab)=-\hsi(a)-\hsi(b)-1$ if the element~$ab$ is of order~$2$.

% Picture

\begin{figure}
  \begin{center}
    \forcehmode
      \bgroup
        \beginpicture
          \setcoordinatesystem units <25 bp,25 bp>
          \multiput {\phantom{$\bullet$}} at -2 0 2 2 /
          \plot 0 0 -2 1 /
          \arrow <7pt> [0.2,0.5] from -2 1 to -1 0.5
          \plot 0 0 -1 2 /
          \arrow <7pt> [0.2,0.5] from 0 0 to -0.5 1
          \plot 0 0 1 2 /
          \arrow <7pt> [0.2,0.5] from 1 2 to 0.5 1
          \plot 0 0 2 1 /
          \arrow <7pt> [0.2,0.5] from 0 0 to 1 0.5
          \put {$a$} [br] <0pt,2pt> at -2 1
          \put {$a$} [br] <0pt,2pt> at -1 2
          \put {$b$} [bl] <0pt,2pt> at 2 1
          \put {$b$} [bl] <0pt,2pt> at 1 2
        \endpicture
      \egroup
  \end{center}
  \caption{$\hsi(ab)=\hsi(a)+\hsi(b)+1$}
  \label{fig-pos-pair}
\end{figure}

\item
if $ab$ is in~$\piorbO(P,p)$
and the elements~$a$ and~$b$ can be represented by a pair of simple contours in $P$ intersecting in exactly one point~$p$ with $\<a,b\>=0$
and placed in a neighbourhood of the point~$p$ as shown in Figure~\ref{fig-neg-pair},
then~$\hsi(ab)=\hsi(a)+\hsi(b)-1$.
\item
For any standard basis
$$v=\{a_1,b_1,\dots,a_g,b_g,c_{g+1},\dots,c_n)\}$$
of $\piorb(P,p)$ we have
$$\sum\limits_{i=g+1}^n\,\hsi(c_i)\equiv(2-2g)-(n-g)\mod\,m.$$
\item
For any elliptic element~$c_{g+l_h+l_p+i}$, $i=1,\dots,l_e$, we have $p_i\cdot\hsi(c_{g+l_h+l_p+i})+1\equiv0\mod\,m$.
\end{enumerate}
\end{lem}

\begin{proof}
According to Theorem~\ref{thm2.1inN}
either the set
$$V=\{\Phi^{-1}(a),\Phi^{-1}(b),\Phi^{-1}(c)\}$$
or the set
$$V^{-1}=\{\Phi^{-1}(a^{-1}),\Phi^{-1}(b^{-1}),\Phi^{-1}(c^{-1})\}$$
is sequential.
This sequential set can be of signature $(0:*,*,*)$ or $(1:*)$.
\begin{enumerate}[$\bullet$]
\item
If $V$ is a sequential set of signature $(1:*)$, then according to Lemma~\ref{lem-product-of-crossing-hyps} we obtain
$$\hsi(ab)=\hsi(a)+\hsi(b).$$
\item
If $V$ is a sequential set of signature $(0:*,*,*)$,
then according to Corollary~\ref{cor-product-dont-intersect} we obtain
$$\hsi(ab)=\hsi(a)+\hsi(b)+1$$
if the element~$ab$ is not of order~$2$ and
$$\hsi(ab)=-\hsi(a)-\hsi(b)-1$$
if the element~$ab$ is of order~$2$.
\item
If $V^{-1}$ is a sequential set of signature $(0:*,*,*)$, then according to Corollary~\ref{cor-product-dont-intersect} we obtain,
$$\hsi(b^{-1}a^{-1})=\hsi(a^{-1})+\hsi(b^{-1})+1$$
if the element~$ab$ is not of order~$2$ and
$$\hsi(b^{-1}a^{-1})=-\hsi(a^{-1})-\hsi(b^{-1})-1$$
if the element~$ab$ is of order~$2$.
Therefore for the element~$ab$ not of order~$2$ we obtain
\begin{align*}
  \hsi(ab)
  &=-\hsi((ab)^{-1})
  =-\hsi(b^{-1}a^{-1})\\
  &=-(\hsi(a^{-1})+\hsi(b^{-1})+1)\\
  &=-\hsi(a^{-1})-\hsi(b^{-1})-1\\
  &=\hsi(a)+\hsi(b)-1
\end{align*}
and for the element~$ab$ of order~$2$ we obtain
\begin{align*}
  \hsi(ab)
  &=\hsi((ab)^{-1})
  =\hsi(b^{-1}a^{-1})\\
  &=-\hsi(a^{-1})-\hsi(b^{-1})-1\\
  &=\hsi(a)+\hsi(b)-1.
\end{align*}
\end{enumerate}
To prove properties~4 and~5 of $\hsi$ we apply Lemma~\ref{lem-lift-glm}.
\end{proof}

We now formalize the properties of the function $\hsi$ in the following definition:

\begin{mydef}
\label{def-m-arf}
We denote by $\piorbO(P,p)$ the set of all non-trivial elements of $\piorb(P,p)$ that can be represented by simple contours.
%that either do not belong to the kernel of the intersection form or are homologous to a hole or a puncture.
An {\it $m$-Arf function\/} is a function
$$\arf:\piorbO(P,p)\to\z/m\z$$
satisfying the following conditions
\begin{enumerate}[1.]
\item
\label{arf-prop-conj}
$\arf(bab^{-1})=\arf(a)$ for any elements~$a,b\in\piorbO(P,p)$,
% such that the element~$bab^{-1}$ is in~$\piorbO(P,p)$,
\item
\label{arf-prop-inv}
$\arf(a^{-1})=-\arf(a)$ for any element~$a\in\piorbO(P,p)$ that is not of order~$2$,
\item
\label{arf-prop-cross}
$\arf(ab)=\arf(a)+\arf(b)$
for any
% elements~$a,b\in\piorbO(P,p)$ such that the element~$ab$ is in~$\piorbO(P,p)$ and the
elements~$a$ and~$b$ which can be represented by a pair of simple contours in $P$
intersecting in exactly one point~$p$ with $\<a,b\>\ne0$,
\item
\label{arf-prop-neg-seqset}
%\begin{enumerate}[(i)]
%\item
$\arf(ab)=\arf(a)+\arf(b)-1$
for any elements~$a,b\in\piorbO(P,p)$ such that the element~$ab$ is in~$\piorbO(P,p)$ %and not of order~$2$
and the elements~$a$ and~$b$ can be represented by a pair of simple contours in $P$
intersecting in exactly one point~$p$ with $\<a,b\>=0$
and placed in a neighbourhood of the point~$p$ as shown in Figure~\ref{fig-neg-pair}.
%\item
%$\arf(ab)=-\arf(a)-\arf(b)-1$
%for any elements~$a,b\in\piorbO(P,p)$ such that the element~$ab$ is in~$\piorbO(P,p)$ and of order~$2$
%and the elements~$a$ and~$b$ can be represented by a pair of simple contours in $P$
%intersecting in exactly one point~$p$ with $\<a,b\>=0$
%and placed in a neighbourhood of the point~$p$ as shown in Figure~\ref{fig-neg-pair}.
%%\ie in such a way that the oriented contours $a$, $b$, and
%%$(ab)^{-1}$ are freely homotopic to pairwise non-intersecting simple contours
%%with orientation opposite to the one induced by the complex structure of the sphere with three
%%holes that they cut out of $P$.
%\end{enumerate}
\item
For any elliptic element~$c$ of order~$p$ we have $p\cdot\arf(c)+1\equiv0\mod\,m$.
\end{enumerate}
% As for the function~$\hsi_{\lats}$, all equations involving~$\arf$ are to be understood as equations in~$\z/m\z$.
\end{mydef}

The following property of $m$-Arf functions follows immediately
from Properties~\ref{arf-prop-neg-seqset} and~\ref{arf-prop-inv} in Definition~\ref{def-m-arf}:

\begin{prop}
\label{arf-prop-seqset}
Let $a$ and~$b$ be elements of~$\piorbO(P,p)$ such that the element~$ab$ is in~$\piorbO(P,p)$
and the elements~$a$ and~$b$ can be represented by a pair of simple contours in $P$
intersecting in exactly one point~$p$ with $\<a,b\>=0$
and placed in a neighbourhood of the point~$p$ as shown in Figure~\ref{fig-pos-pair}.
Then the equation~$\arf(ab)=\arf(a)+\arf(b)+1$ is satisfied if the element~$ab$ is not of order~$2$
and the equation~$\arf(ab)=-\arf(a)-\arf(b)-1$ is satisfied if the element~$ab$ is of order~$2$.
\end{prop}

\begin{lem}
\label{c1c2-not-ell}
Let $\Gamma$ be a hyperbolic polygon group of signature $(0:p_1,\dots,p_r)$, $r>3$.
Let $c_1,\dots,c_r$ be a standard basis of~$\Gamma$.
Then the element~$c_1c_2$ is not elliptic.
\end{lem}

\begin{proof}
Let $\Pi$~be the canonical fundamental polygon for the group generated by the elements $c_1,\dots,c_n$
such that the generators~$c_i$ can be described as products $c_i=\si_i\si_{i+1}$
of reflexions~$\si_1,\dots,\si_n$ in the edges of the polygon~$\Pi$ (suitably numbered).
Then $c_1c_2=(\si_1\si_2)(\si_2\si_3)=\si_1\si_3$.
The product of two reflexions~$\si_1\si_2$ is an elliptic element
if and only if the axes of the reflexions intersect in~$\hyp$.
Since~$r>3$,
the sides of the polygon~$\Pi$ that correspond to the reflexions~$\si_1$ and~$\si_3$ are not next to each other.
Let us assume that the axes intersect and let~$Q$ be the hyperbolic polygon
enclosed between by the axes and the polygon~$\Pi$.
All angles of the polygon~$\Pi$ are acute.
One angle of the polygon~$Q$ is the angle between the intersecting axes,
two angles are larger than~$\pi/2$, all other angles of~$Q$ are larger than~$\pi$,
hence the sum of the angles of~$Q$ is larger that it should be for a hyperbolic polygon.
\end{proof}

\begin{prop}
\label{arf-prop-liftexist}
For any standard basis
$$v=\{a_1,b_1,\dots,a_g,b_g,c_{g+1},\dots,c_{g+l_h+l_p+l_e}\}$$
of $\piorb(P,p)$ we have
$$\sum\limits_{j=g+1}^n\,\arf(c_j)\equiv(2-2g)-(l_h+l_p+l_e)\mod\,m.$$
\end{prop}

\begin{proof}
We discuss the case $g=0$ first, and then we reduce the general case to the case $g=0$.
\begin{enumerate}[$\bullet$]
\item
Let $g=0$.
We prove that the statement is true for lifts of sequential sets of signature $(0:p_1,\dots,p_r)$ by induction on~$r$.

\bigskip
In the case $r=3$ Proposition~\ref{arf-prop-seqset} implies
$$\arf(c_1c_2)=\arf(c_1)+\arf(c_2)+1$$
if the element~$c_1c_2=c_3^{-1}$ is not of order~$2$ and
$$\arf(c_1c_2)=-\arf(c_1)-\arf(c_2)-1$$
if the element~$c_1c_2=c_3^{-1}$ is of order~$2$.
If the element~$c_3$ is not of order~$2$, then Property~\ref{arf-prop-inv} implies
$$\arf(c_1c_2)=\arf(c_3^{-1})=-\arf(c_3).$$
Combining~$\arf(c_1c_2)=\arf(c_1)+\arf(c_2)+1$ and~$\arf(c_1c_2)=-\arf(c_3)$, we obtain
$$\arf(c_1)+\arf(c_2)+\arf(c_3)=-1.$$
If the element~$c_3$ is of order~$2$, then
$$\arf(c_1c_2)=\arf(c_3^{-1})=\arf(c_3).$$
Combining~$\arf(c_1c_2)=-\arf(c_1)-\arf(c_2)-1$ and~$\arf(c_1c_2)=\arf(c_3)$, we obtain
$$\arf(c_1)+\arf(c_2)+\arf(c_3)=-1.$$

\bigskip
Assume that the statement is true for $r\le k-1$ and consider the case $r=k$.
By our assumption
$$\arf(c_1\cdot c_2)+\arf(c_3)+\cdots+\arf(c_k)=2-(k-1)=(2-k)+1.$$
Moreover, according to Lemma~\ref{c1c2-not-ell} the element~$c_1c_2$ cannot be of order~$2$.
Hence by Proposition~\ref{arf-prop-seqset} we have $\arf(c_1c_2)=\arf(c_1)+\arf(c_2)+1$.
The last two equations imply that $\arf(c_1)+\cdots+\arf(c_k)=2-k$.
\item
We now consider the general case.
The set
$$(a_1,b_1a_1^{-1}b_1^{-1},\dots,a_g,b_ga_g^{-1}b_g^{-1},c_{g+1},\dots,c_{g+r})$$
is a standard basis of an orbifold of signature $(0:2g+l_h,l_p,l_e:p_1,\dots,p_{l_e})$, hence
\begin{align*}
  &\sum\limits_{i=1}^g\,(\arf(a_i)+\arf(b_ia_i^{-1}b_i^{-1}))+\sum\limits_{i=g+1}^{g+l_h+l_p+l_e}\arf(c_i)\\
  &=2-(2g+l_h+l_p+l_e)
  =(2-2g)-(l_h+l_p+l_e).
\end{align*}
From Properties~\ref{arf-prop-conj} and~\ref{arf-prop-inv} of $m$-Arf functions we obtain that
$\arf(b_i a_i^{-1}b_i^{-1})=\arf(a_i^{-1})=-\arf(a_i)$ and hence $\arf(a_i)+\arf(b_i a_i^{-1}b_i^{-1})=0$.
\end{enumerate}
\end{proof}

\begin{mydef}
\label{def-si-to-lats}
Let $\hsi_{\lats}:\piorb(P,p)\to\z/m\z$ be the function associated to a lift $\lats$ as in definition~\ref{def-hsi-to-lats},
then the function $\si_{\lats}=\hsi_{\lats}|_{\piorbO(P,p)}$
is an $m$-Arf function according to
Lemma~\ref{lem-hsi-rules}, \ref{lem-inv}, and~\ref{lem-conj}.
We call the function $\si_{\lats}$ the {\it $m$-Arf function associated to the lift $\lats$.}
\end{mydef}

\subsection{Higher Arf functions and autohomeomorphisms of orbifolds}

\label{arf-propert}

\myskip
Let $\lat$ be a Fuchsian group of signature~$(g:p_1,\dots,p_r)$ and $P=\hyp/\lat$ the corresponding orbifold.
Let $p\in P$.
Let $\Psi:\hyp\to P$ be the natural projection.
Choose $q\in\Psi^{-1}(p)$ and let $\Phi:\lat\to\piorbO(P,p)$ be the induced isomorphism.
Let $\lats$ be a lift of $\lat$ in $\Gm$.

\myskip
%We recall the definition of the Dehn twists (\cite{Dehn}, \cite{Nbook} Chapter~1, Lemma~7.4).
Consider the following transformations of a standard basis:
$$v=\{a_1,b_1,\dots,a_g,b_g,c_{g+1},\dots,c_n\}$$
of $\piorbO(P,p)$ to another standard basis
$$v'=\{a_1',b_1',\dots,a_g',b_g',c_{g+1}'\dots,c_n'\}:$$
\begin{align*}
1.\quad &a_1'=a_1b_1.\\
2.\quad &a_1'=(a_1a_2)a_1(a_1a_2)^{-1},\\
        &b_1'=(a_1a_2)a_1^{-1}a_2^{-1}b_1(a_1a_2)^{-1},\\
        &a_2'=a_1a_2a_1^{-1},\\
        &b_2'=b_2a_2^{-1}a_1^{-1}.\\
3.\quad &a_g'=(b_g^{-1}c_{g+1})b_g^{-1}(b_g^{-1}c_{g+1})^{-1},\\
        &b_g'=(b_g^{-1}c_{g+1}b_g)c_{g+1}^{-1}b_ga_gb_g^{-1}(b_g^{-1}c_{g+1}b_g)^{-1},\\
        &c_{g+1}'=b_g^{-1}c_{g+1}b_g.\\
4.\quad &a_k'=a_{k+1},\quad b_k'=b_{k+1},\\
        &a_{k+1}'=(c_{k+1}^{-1}c_k)a_k(c_{k+1}^{-1}c_k)^{-1},\\
        &b_{k+1}'=(c_{k+1}^{-1}c_k)b_k(c_{k+1}^{-1}c_k)^{-1}.\\
5.\quad &c_k'=c_{k+1},\quad c_{k+1}'=c_{k+1}^{-1}c_kc_{k+1}.
\end{align*}
Here $c_i=[a_i,b_i]$ for $i=1,\dots,g$,
in~4 we consider $k\in\{1,\dots,g\}$,
in~5 we consider $k\in\{g+1,\dots,n\}$ such that~$\ord(c_k)=\ord(c_{k+1})$.
If $a_i'$, $b_i'$ resp.\ $c_i'$ is not described explicitly, this means $a_i'=a_i$, $b_i'=b_i$ resp.\ $c_i'=c_i$.

\bigskip
We will call these transformations generalised Dehn twists.
Each generalised Dehn twists induces a homotopy class of autohomeomorphisms of the orbifold~$P$,
which maps elliptic fixed points to elliptic fixed points of the same order.
%holes to holes and punctures to punctures.
The group of all homotopy classes of autohomeomorphisms of the orbifold~$P$
is generated by the homotopy classes of generalised Dehn twists as described above (compare~\cite{Zieschang:1973}).

\bigskip
Now we will compute the values of an Arf function~$\arf$ on the standard basis~$v'$ from the values of $\arf$ on the standard basis~$v$ for each of the generalised Dehn twists described above.

\begin{lem}
\label{lem-dehn}
Let $\arf:\piorbO(P,p)\to\z/m\z$ be an $m$-Arf function.
Let $D$ be a generalised Dehn twist of the type described above.
Suppose that $D$ maps the standard basis
$$v=\{a_1,b_1,\dots,a_g,b_g,c_{g+1},\dots,c_n\}$$
into the standard basis
$$v'=D(v)=\{a_1',b_1',\dots,a_g',b_g',c_{g+1}',\dots,c_n'\}.$$
Let $\al_i,\be_i,\ga_i$ resp.\ $\al'_i,\be'_i,\ga'_i$ be the values of $\arf$ on the elements of $v$ resp.\ $v'$.
Then for the Dehn twists of types~1--5 we obtain
\begin{align*}
1.\quad &\al_1'=\al_1+\be_1.\\
2.\quad &\be_1'=\be_1-\al_1-\al_2-1,\quad \be_2'=\be_2-\al_2-\al_1-1.\\
3.\quad &\al_g'=-\be_g,\quad \be_g'=\al_g-\ga_{g+1}-1.\\
4.\quad &\al_k'=\al_{k+1},\quad \be_k'=\be_{k+1},\quad \al_{k+1}'=\al_k,\quad \be_{k+1}'=\be_k.\\
5.\quad &\ga_k'=\ga_{k+1},\quad \ga_{k+1}'=\ga_k.
\end{align*}
\end{lem}

\begin{proof}
We assume that the Dehn twist $D$ belongs to one of the types described in the definition above.
In the following computations we illustrate the position of the contours on the surface
with figures showing the position of the axes of the corresponding elements in $\lat$.
Let
$$\{A_1,B_1,\dots,A_g,B_g,C_{g+1},\dots,C_n\}$$
be the sequential set corresponding to the standard basis~$v$.
In the first case according to Property~\ref{arf-prop-cross} of $m$-Arf functions
we obtain
$$\arf(a_1')=\arf(a_1b_1)=\arf(a_1)+\arf(b_1).$$

\myskip\noindent
In the second case according to Property~\ref{arf-prop-conj} we obtain
\begin{align*}
  \arf(a_1')&=\arf((a_1a_2)a_1(a_1a_2)^{-1})=\arf(a_1),\\
  \arf(b_1')&=\arf((a_1a_2)a_1^{-1}a_2^{-1}b_1(a_1a_2)^{-1})=\arf(a_1^{-1}a_2^{-1}b_1)\\
            &=\arf(a_1(a_1^{-1}a_2^{-1}b_1)a_1^{-1})=\arf(a_2^{-1}b_1a_1^{-1}).
\end{align*}

% Picture
\begin{figure}
  \begin{center}
    \forcehmode
      \bgroup
        \beginpicture
          \setcoordinatesystem units <20 bp,20 bp>
          \multiput {\phantom{$\bullet$}} at -4 -1 10 3 /
          \circulararc 180 degrees from 1 0 center at -1 0
          \arrow <7pt> [0.2,0.5] from 0.995 0.04 to 1 0
          \put {$A_1^{-1}$} [b] <0pt,\baselineskip> at -1 2
          \circulararc 180 degrees from 4 0 center at 2 0
          \arrow <7pt> [0.2,0.5] from 3.995 0.04 to 4 0
          \put {$B_1$} [b] <0pt,\baselineskip> at 2 2
          \circulararc 180 degrees from 2.5 0 center at 0.5 0
          \arrow <7pt> [0.2,0.5] from 2.495 0.04 to 2.5 0
          \put {$B_1A_1^{-1}$} [b] <0pt,\baselineskip> at 0.5 2
          \circulararc 180 degrees from 9 0 center at 7 0
          \arrow <7pt> [0.2,0.5] from 8.995 0.04 to 9 0
          \put {$A_2^{-1}$} [b] <0pt,\baselineskip> at 7 2
          \plot -4 0 10 0 /
        \endpicture
      \egroup
  \end{center}
  \caption{Axes of $B_1A_1^{-1}$ and $A_2^{-1}$}
  \label{fig-axes-1}
\end{figure}

\noindent
The mutual position of the axes of the elements~$A_2^{-1}$ and~$B_1A_1^{-1}$ is as in Figure~\ref{fig-axes-1},
hence Property~\ref{arf-prop-neg-seqset} implies
$$
  \arf(b_1')
  =\arf(a_2^{-1}\cdot(b_1a_1^{-1}))
  =\arf(a_2^{-1})+\arf(b_1a_1^{-1})-1.
$$
According to Property~\ref{arf-prop-cross} we have $\arf(b_1a_1^{-1})=\arf(b_1)+\arf(a_1^{-1})$.
Thus using Property~\ref{arf-prop-inv} we obtain
$$\arf(b_1')=\arf(a_2^{-1})+\arf(b_1)+\arf(a_1^{-1})-1=\arf(b_1)-\arf(a_1)-\arf(a_2)-1.$$
%According to Property~\ref{arf-prop-conj} we obtain
%\begin{align*}
%  \arf(a_2')&=\arf(a_1a_2a_1^{-1})=\arf(a_2),\\
%  \arf(b_2')&=\arf(b_2a_2^{-1}a_1^{-1}).
%\end{align*}
%Property~\ref{arf-prop-neg-seqset} implies
%$$
%  \arf(b_2')
%  =\arf((b_2a_2^{-1})\cdot a_1^{-1})
%  =\arf(b_2a_2^{-1})+\arf(a_1^{-1})-1.
%$$
%According to Property~\ref{arf-prop-cross} we have $\arf(b_2a_2^{-1})=\arf(b_2)+\arf(a_2^{-1})$.
%Thus using Property~\ref{arf-prop-inv} we obtain
%$$\arf(b_2')=\arf(b_2)+\arf(a_2^{-1})+\arf(a_1^{-1})-1=\arf(b_2)-\arf(a_1)-\arf(a_2)-1.$$
Similarly we show that $\arf(a_2')=\arf(a_2)$ and $\arf(b_2')=\arf(b_2)-\arf(a_2)-\arf(a_1)-1$.

\myskip\noindent
In the third case we obtain according to Properties~\ref{arf-prop-inv} and~\ref{arf-prop-conj}
\begin{align*}
  \arf(a_g')&=\arf((b_g^{-1}c_{g+1})b_g^{-1}(b_g^{-1}c_{g+1})^{-1})=\arf(b_g^{-1})=-\arf(b_g),\\
  \arf(b_g')&=\arf((b_g^{-1}c_{g+1}b_g)c_{g+1}^{-1}b_ga_gb_g^{-1}(b_g^{-1}c_{g+1}b_g)^{-1})=\arf(c_{g+1}^{-1}b_ga_gb_g^{-1}),\\
  \arf(c_{g+1}')&=\arf(b_g^{-1}c_{g+1}b_g)=\arf(c_{g+1}).
\end{align*}

% Picture
\begin{figure}
  \begin{center}
    \forcehmode
      \bgroup
        \beginpicture
          \setcoordinatesystem units <15 bp,15 bp>
          \multiput {\phantom{$\bullet$}} at -10 -1 8 3 /
          \circulararc 180 degrees from -5 0 center at -7 0
          \arrow <7pt> [0.2,0.5] from -8.995 0.04 to -9 0
          \put {$A_g$} [b] <0pt,\baselineskip> at -7 2
          \circulararc 180 degrees from -2 0 center at -4 0
          \arrow <7pt> [0.2,0.5] from -2.005 0.04 to -2 0
          \put {$B_g$} [b] <0pt,\baselineskip> at -4 2
          \circulararc 180 degrees from 1 0 center at -1 0
          \arrow <7pt> [0.2,0.5] from 0.995 0.04 to 1 0
          \put {$B_gA_gB_g^{-1}$} [b] <0pt,\baselineskip> at -1 2
          \circulararc 180 degrees from 7 0 center at 5 0
          \arrow <7pt> [0.2,0.5] from 6.995 0.04 to 7 0
          \put {$C_{g+1}^{-1}$} [b] <0pt,\baselineskip> at 5 2
          \plot -10 0 8 0 /
        \endpicture
      \egroup
  \end{center}
  \caption{Axes of $C_{g+1}^{-1}$ and $B_gA_gB_g^{-1}$}
  \label{fig-axes-2}
\end{figure}

\noindent
The mutual position of the axes of the elements $C_{g+1}^{-1}$ and $B_gA_gB_g^{-1}$ is as in Figure~\ref{fig-axes-2}.
According to Properties~\ref{arf-prop-neg-seqset} and~\ref{arf-prop-conj} we obtain
$$
  \arf(b_g')
  =\arf(c_{g+1}^{-1}\cdot(b_ga_gb_g^{-1}))
  =\arf(c_{g+1}^{-1})+\arf(b_ga_gb_g^{-1})-1
  =\arf(c_{g+1}^{-1})+\arf(a_g)-1.
$$

\myskip\noindent
In the forth and fifth case computations are easy, we only use Property~\ref{arf-prop-conj} of $m$-Arf functions.
\end{proof}

\subsection{Correspondence between higher Arf functions and hyperbolic Gorenstein automorphy factors}

\myskip
Let $\lat$ be a Fuchsian group of signature~$(g:p_1,\dots,p_r)$ and $P=\hyp/\lat$ the corresponding orbifold.
Let $p\in P$.
Let $\Psi:\hyp\to P$ be the natural projection.
Choose $q\in\Psi^{-1}(p)$ and let $\Phi:\lat\to\piorbO(P,p)$ be the induced isomorphism.

\begin{lem}
\label{lem-arf-difference}
The difference $\arf_1-\arf_2:\piorbO(P,p)\to\z/m\z$ of two Arf functions~$\arf_1$ and~$\arf_2$
induces a linear function $\ell:H_1(P;\z/m\z)\to\z/m\z$.
\end{lem}

\begin{proof}
The proof is analogous to the proof of the corresponding statement
for higher Arf functions on Fuchsian groups without torsion (see Lemma~4.5 in~\cite{NP1}).
The main observation is the fact that according to Lemma~\ref{lem-dehn}
the action of the generalised Dehn twists on the tuples of values of a higher Arf function
on elements of a standard basis are by affine-linear maps,
therefore the action on the tuples of differences of values of two higher Arf functions is by linear maps.
\end{proof}

\begin{cor}
\label{cor-arf-affine}
The set $\Arf^{P,m}$ of all $m$-Arf functions on $\piorbO(P,p)$ has a structure of an affine space,
\ie the set $\{\arf-\arf_0\st\arf\in\Arf^{P,m}\}$ is a free module over~$\z/m\z$ for any $\arf_0\in\Arf^{P,m}$.
\end{cor}

\begin{cor}
\label{cor-arf-def-by-gen}
An $m$-Arf function is uniquely determined by its values on the elements of some standard basis of $\piorbO(P,p)$.
\end{cor}

\begin{thm}
\label{thm-corresp}
Let $\lat$ be a Fuchsian group of signature~$(g:p_1,\dots,p_r)$ and $P=\hyp/\lat$ the corresponding orbifold.
Let $p\in P$.
There is a 1-1-correspondence between
\begin{enumerate}[1)]
\item hyperbolic Gorenstein automorphy factors of level~$m$ associated to the Fuchsian group~$\lat$.
\item lifts of~$\lat$ into $\Gm$.
\item $m$-Arf functions $\arf:\piorbO(P,p)\to\z/m\z$.
% this is now a part of the def of an Arf function: such that the value of~$\arf$ on an element~$c$ of order~$p$ satisfies the condition~$p\cdot\arf(c)+1\equiv0\mod~m$.
\end{enumerate}
\end{thm}

\begin{proof}
According to Proposition~\ref{autfac-lifts} there is a 1-1-correspondence
between hyperbolic Gorenstein automorphy factors of level~$m$ associated to the Fuchsian group~$\lat$ and the lifts of $\lat$ into $\Gm$.
In Definition~\ref{def-hsi-to-lats} we attached to any lift $\lats$ of $\lat$ into $\Gm$ an $m$-Arf function $\si_{\lats}$ on~$P$.
On the other hand we can attach to any $m$-Arf function $\arf$ a subset of~$\Gm$
$$\lats_{\arf}=\{g\in\Gm\st\pi(g)\in\lat,~\levm(g)=\arf(\Phi(\pi(g)))\},$$
where $\pi:\Gm\to G$ is the covering map.
It remains to prove that this subset of~$\Gm$ is actually a lift of~$\lat$.
Let
$$v=\{a_1,b_1,\dots,a_g,b_g,c_{g+1},\dots,c_{g+r}\}=\{d_1,\dots,d_{2g+r}\}$$
be a standard basis of $\piorb(P,p)$
and let~$\bV=\{\Phi^{-1}(d_1),\dots,\Phi^{-1}(d_{2g+r})\}$ be the corresponding sequential set.
Let~$\{D_j\}_{j=1,\dots,2g+r}$ be a lift of the sequential set~$\bV$,
i.e.\ $\pi(D_j)=\Phi^{-1}(d_j)$, such that $\levm(D_j)=\arf(d_j)$.
Then we obtain according to Proposition~\ref{arf-prop-liftexist} that
$$\sum\limits_{i=g+1}^n\,\levm(C_i)=\sum\limits_{i=g+1}^n\,\arf(c_i)=(2-2g)-(n-g)\mod\,m,$$
hence by Lemma~\ref{lem-lift-glm} we obtain
$$\prod\limits_{i=1}^g\,[A_i,B_i]\cdot\prod\limits_{i=g+1}^n\,C_i=e.$$
This and the fact that for any~$i=1,\dots,r$
$$
  p_i\cdot\levm(C_{g+i})+1
  =p_i\cdot\arf(c_{g+i})+1
  \equiv0\mod~m
$$
imply according to Lemma~\ref{lem-liftonerel} that
the subgroup $\lats$ of $\Gm$ generated by $V$ is a lift of $\lat$ into $\Gm$.
Let us compare the corresponding Arf function $\arf_{\lats}$ with the Arf function $\arf$.
We have
$$\arf_{\lats}(d_j)=\levm(D_j)=\arf(d_j)$$
for all~$j$ \ie the Arf functions $\arf_{\lats}$ and $\arf$ coincide on the standard basis~$v$.
Thus by Lemma~\ref{lem-arf-difference} the Arf functions $\arf_{\lats}$ and $\arf$ coincide on the whole $\piorbO(P,p)$.
From the definition of $\arf_{\lats}$ and $\lats_{\arf}$ we see that this implies that $\lats=\lats_{\arf}$,
hence $\lats_{\arf}$ is indeed a lift of~$\lat$ into~$\Gm$.
It is clear from the definitions that the mappings $\lats\mapsto\arf_{\lats}$ and $\arf\mapsto\lats_{\arf}$
are inverse to each other.
\end{proof}

\begin{cor}
\label{cor-arf-function-exists}
Let $P$ be a Riemann orbifold of signature $(g:p_1,\dots,p_r)$.
Let $v=\{a_1,b_1,\dots,a_g,b_g,c_{g+1},\dots,c_{g+r}\}$ be a standard basis of $\piorb(P,p)$.
An $m$-Arf function on~$\piorbO(P,p)$ exists if and only if the signature $(g:p_1,\dots,p_r)$
satisfies the liftability conditions described in Proposition~\ref{prop-lift-exists}.
Moreover, if the liftability conditions are satisfied then any possible tuple of $2g$ values in~$\z/m\z$
can be realised in a unique way as a set of values on~$a_i,b_i$ of an $m$-Arf function on~$\piorbO(P,p)$,
hence there are $m^{2g}$ different $m$-Arf functions on~$\piorbO(P,p)$.
\end{cor}

\begin{proof}
The statement follows immediately from Theorem~\ref{thm-corresp} and Proposition~\ref{prop-lift-exists}.
\end{proof}

\section{Moduli spaces of Gorenstein singularities}

\label{sec-moduli-spaces}

We study the moduli space of Gorenstein quasi-homogeneous surface singularities (GQHSS).
Using Proposition~\ref{autfac-lifts}, we define the moduli space of GQHSS of level~$m$ as the space of conjugacy classes of subgroups~$\lats$ in~$\Gm$
such that the restriction of the covering map~$\Gm\to G=\PSL$ to $\lats$ is an isomorphism between $\lats$ and a Fuchsian group~$\lat$.
The projection~$\lats\mapsto\lat$ from the moduli space of GQHSS of level~$m$ to the moduli space of Riemann orbifolds is a finite ramified covering.

\myskip
\subsection{Topological classification of higher Arf functions}

\myskip
There is a 1-1-cor\-res\-pon\-dence (see Theorem~\ref{thm-corresp})
between automorphy factors of level~$m$ and $m$-Arf functions on~$\piorbO(P,p)$.
This correspondence allows us to reduce the problem of finding the number of connected components of the moduli space of GQHSS of level~$m$
to the problem of finding the number of orbits of the action of the group of autohomeomorphisms on the set of $m$-Arf functions.
We describe the orbit of an $m$-Arf function
under the action of the group of homotopy classes of surface autohomeomorphisms.

\myskip
Let $P$ be a Riemann orbifold of signature $(g:p_1,\dots,p_r)$.
Let $p\in P$.

\begin{mydef}
\label{def-arf-inv}
Let $\arf:\piorbO(P,p)\to\z/m\z$ be an $m$-Arf function.
We define the {\it Arf invariant\/} $\de=\de(P,\arf)$ of~$\arf$ as follows:
If~$g>1$ and $m$ is even then we set $\de=0$ if there is a standard basis $\{a_1,b_1,\dots,a_g,b_g,c_{g+1},\dots,c_n\}$ of the fundamental group $\piorb(P,p)$ such that
$$\sum\limits_{i=1}^g(1-\arf(a_i))(1-\arf(b_i))\equiv0\mod~2$$
and we set $\de=1$ otherwise.
If~$g>1$ and $m$ is odd then we set $\de=0$.
If~$g=0$ then we set $\de=0$.
If~$g=1$ then we set
$$\de=\gcd(m,p_1-1,\dots,p_r-1,\arf(a_1),\arf(b_1)),$$
%$$\de=\gcd(m,\arf(a_1),\arf(b_1),\arf(c_2)+1,\dots,\arf(c_n)+1),$$
where $\{a_1,b_1,c_2,\dots,c_{r+1}\}$ is a standard basis of the fundamental group $\piorb(P,p)$.
\end{mydef}

\begin{rem}
It is not hard to see that~$\de$ does not change under the transformations described in Lemma~\ref{lem-dehn}, i.e.\ it is indeed an invariant of the Arf function.
\end{rem}

\begin{proof}
%%Let $D$ be a generalised Dehn twist of the type described above.
%%Suppose that $D$ maps the standard basis
%%$$v=\{a_1,b_1,\dots,a_g,b_g,c_{g+1},\dots,c_n\}$$
%%into the standard basis
%%$$v'=D(v)=\{a_1',b_1',\dots,a_g',b_g',c_{g+1}',\dots,c_n'\}.$$
%%Let
%%$$
%%  \{\al_1,\be_1,\dots,\al_g,\be_g,\ga_{g+1},\dots,\ga_n\}
%%  ~\text{resp.}~
%%  \{\al'_1,\be'_1,\dots,\al_g',\be_g',\ga'_{g+1},\dots,\ga_n'\}
%%$$
%%be the values of $\arf$ on the elements of $v$ resp.\ $v'$.
Let $D$, $v$, $v'$, $\al_i$, $\be_i$, $\ga_i$, $\al_i'$, $\be_i'$, $\ga_i'$ be as in Lemma~\ref{lem-dehn}.
Let us first consider the case~$g>1$:
For a Dehn twist of type~1 we have
\begin{align*}
  (1-\al_1')(1-\be_1')
  &=(1-(\al_1+\be_1))(1-\be_1)\\
  &=(1-\al_1)(1-\be_1)-\be_1(1-\be_1)
  \equiv(1-\al_1)(1-\be_1)\mod\,2.
\end{align*}
%since $x(1-x)\equiv0\mod\,2$ for any~$x\in\z$.
For a Dehn twist of type~2 we have
\begin{align*}
  &(1-\al_1')(1-\be_1')+(1-\al_2')(1-\be_2')\\
  &=(1-\al_1)(1-\be_1+\al_1+\al_2+1)+(1-\al_2)(1-\be_2+\al_1+\al_2+1)\\
  &=(1-\al_1)(1-\be_1)+(1-\al_2)(1-\be_2)+(2-(\al_1+\al_2))((\al_1+\al_2)+1)\\
  &\equiv(1-\al_1)(1-\be_1)+(1-\al_2)(1-\be_2)\mod\,2.
\end{align*}
%since $(2-x)(x+1)\equiv0\mod\,2$ for any~$x\in\z$.
For a Dehn twist of type~3, since $m$ is even and~$p_1\cdot\ga_{g+1}+1\equiv0\mod\,m$, we have that $\ga_{g+1}$ is odd.
Then
$$
  (1-\al_g')(1-\be_g')
  =(1+\be_g)(1-\al_g+(\ga_{g+1}+1))
  \equiv(1+\be_g)(1-\al_g)
  \equiv(1-\be_g)(1-\al_g)
$$
since $\ga_{g+1}+1\equiv0\mod\,2$ and~$1+\be_g\equiv1-\be_g\mod\,2$.
Dehn twists of type~4 do not change $\sum_{i=1}^g(1-\arf(a_i))(1-\arf(b_i))$ since they only permute $(\al_k,\be_k)$ with $(\al_{k+1},\be_{k+1})$.
Dehn twists of type~5 do not change $\sum_{i=1}^g(1-\arf(a_i))(1-\arf(b_i))$ since they only permute $\ga_i$.

Let us now consider the case~$g=1$:
%Let $D$, $v$, $v'$, $\al_i$, $\be_i$, $\ga_i$, $\al_i'$, $\be_i'$, $\ga_i'$ be as in Lemma~\ref{lem-dehn}.
Dehn twists of types~2 and~4 involve pairs~$a_i,b_i$ and~$a_j,b_j$, i.e.\ they are not applicable in the case~$g=1$.
Dehn twist of type~5 only swaps the orders of two elliptic fixed points, hence it does not change~$\de$.
For a Dehn twist of type~1 we obtain $\al_1'=\al_1+\be_1$ and $\be_1'=\be_1$.
Thus
$$\gcd(\al_1',\be_1')=\gcd(\al_1+\be_1,\be_1)=\gcd(\al_1,\be_1)$$
and therefore $\gcd(m,p_1-1,\dots,p_r-1,\al_1',\be_1')=\gcd(m,p_1-1,\dots,p_r-1,\al_1,\be_1)$.
For a Dehn twist of type~3 we obtain $\al_1'=-\be_1$ and $\be_1'=\al_1-\ga_2-1$.
Let~$d$ be a common divisor of $m,p_1-1,\dots,p_r-1,\al_1,\be_1$, i.e
$$m\equiv\al_1\equiv\be_1\equiv0\mod\,d,\quad p_1\equiv\cdots\equiv p_r\equiv1\mod\,d.$$
We know that $p_1\cdot\ga_2+1\equiv0\mod\,m$, but~$m\equiv0\mod\,d$, hence $p_1\cdot\ga_2+1\equiv0\mod\,d$.
Since~$p_1\equiv1\mod\,d$, we obtain that $\ga_2+1\equiv0\mod\,d$.
Hence $d$ is a common divisor of $m,p_1-1,\dots,p_r-1,\al_1'=-\be_1,\be_1'=\al_1-(\ga_2+1)$.
Similarly every common divisor of $m,p_1-1,\dots,p_r-1,\al_1',\be_1'$
is a common divisor of $m,p_1-1,\dots,p_r-1,\al_1,\be_1$.
Thus
$$\gcd(m,p_1-1,\dots,p_r-1,\al_1',\be_1')=\gcd(m,p_1-1,\dots,p_r-1,\al_1,\be_1).\qedhere$$
\end{proof}

\begin{mydef}
\label{def-type-hyp-arf-func}
By the {\it type of the $m$-Arf function\/}~$(P,\arf)$ we mean the tuple
$$(g,p_1,\dots,p_r,\de),$$
where $\de$ is the Arf invariant of $\arf$ defined above.
\end{mydef}

\begin{lem}
\label{lem-normalise}
Let $\arf:\piorbO(P,p)\to\z/m\z$ be an $m$-Arf function.
\begin{enumerate}[(a)]
\item
If~$g>1$ then there is a standard basis $v=\{a_1,b_1,\dots,a_g,b_g,c_{g+1},\dots,c_n\}$
of $\piorb(P,p)$ such that
$$(\arf(a_1),\arf(b_1),\dots,\arf(a_g),\arf(b_g))=(0,\xi,1,\dots,1)$$
with $\xi\in\{0,1\}$.
If $m$ is odd then the basis can be chosen in such a way that $\xi=1$, i.e. so that
$$(\arf(a_1),\arf(b_1),\dots,\arf(a_g),\arf(b_g))=(0,1,1,\dots,1)$$
\item
If~$g=1$ then there is a standard basis $v=\{a_1, b_1, c_2,\dots,c_n\}$ of $\piorb(P,p)$ such that $(\arf(a_1),\arf(b_1))=(\de,0)$, where $\de$ is the Arf invariant of~$\arf$.
\end{enumerate}
\end{lem}

\begin{proof}
The proof is along the lines of the proofs of Lemma~5.1 and Lemma~5.2 in~\cite{NP1}.
Using generalised Dehn twists of types~1,2 and~4 we can show that a basis can be chosen in the desired way.
The last step in the proof of Lemma~5.1 in~\cite{NP1} was to show
that if $m$ is even and~$\arf(c_{g+1})$ is even
then we can transform a basis with
$$(\arf(a_1),\arf(b_1),\dots,\arf(a_g),\arf(b_g))=(0,0,1,\dots,1)$$
into the basis with
$$(\arf(a_1),\arf(b_1),\dots,\arf(a_g),\arf(b_g))=(0,1,1,\dots,1).$$
However in the situation we are considering now we know
that $\arf(c_i)$ satisfies the equation~$p_i\cdot\arf(c_i)+1\equiv0\mod~m$.
Therefore if $m$ is even then $\arf(c_i)$ must be odd.
Hence this last reduction step does not apply in the case considered here.
\end{proof}

\begin{rem}
An autohomeomorphism of a surface $P$ induces an automorphism of the lifted Fuchsian group.
Let $\calA$ be the corresponding group of such automorphisms of lifts of Fuchsian groups.
On the other hand, any autohomeomorphism generates an element of~$\MathOpSp(2g,\z)$, where $g$ is the genus of~$P$.
Lemma~\ref{lem-normalise} implies that for two autohomeomorphisms
the corresponding elements in~$\calA$ differ if the corresponding elements in~$\MathOpSp(2g,\z_m)$ differ.
Thus we obtain a homomorphism~$f:\calA\to\MathOpSp(2g,\z_m)$.
Using generalised Dehn twists of types~1,2 and~4 we can show that $f$ is an epimorphism.
Lemma~\ref{lem-normalise} implies that $\ker(f)=d\cdot T$,
where $T$ is the group of all parallel translations on the affine space of all lifts.
Using Dehn twists of types~1--5 and Lemma~\ref{lem-normalise} we are able to determine the number~$d$.
If~$g>1$ then~$d=2$ if $m$ is even and $d=1$ otherwise.
If~$g=1$ then $d=\gcd(m,p_1-1,\dots,p_r-1,\arf(a_1),\arf(b_1))$.
\end{rem}

\begin{thm}
\label{thm-is-type}
A tuple $t=(g,p_1,\dots,p_r,\de)$ is the type of a hyperbolic $m$-Arf function
on a Riemann orbifold of signature $(g:p_1,\dots,p_r)$ if and only if it has the following properties:
\begin{enumerate}[(a)]
\item
The liftability conditions:
The orders~$p_1,\dots,p_r$ are prime with~$m$ and satisfy the condition
$$(p_1\cdots p_r)\cdot\left(\sum\limits_{i=1}^r\,\frac1{p_i}-(2g-2)-r\right)\equiv0\mod m.$$
\item
If $g>1$ then $\de\in\{0,1\}$.
\item
If $g>1$ and $m$ is odd then $\de=0$.
\item
If $g=1$ then $\de$ is a divisor of~$\gcd(m,p_1-1,\dots,p_r-1)$.
\item
If~$g=0$ then $\de=0$.
\end{enumerate}
\end{thm}

\begin{proof}
%We will discuss the case~$g>1$, the case~$g=1$ is similar.
Let us first assume that the tuple $t$ is a type of a hyperbolic $m$-Arf function on a orbifold of signature $(g:p_1,\dots,p_r)$.
Then according to Corollary~\ref{cor-arf-function-exists} the signature $(g:p_1,\dots,p_r)$ satisfies the liftability conditions.
%and that $\arf$ is such an $m$-Arf function.
%Then according to Proposition~\ref{arf-prop-liftexist} we have $p_i\cdot\arf(c_{g+i})+1\equiv0\mod\,m$ for~$i=1,\dots,r$,
%hence $p_i$ is prime with~$m$.
%Furthermore
%$$(2g-2)+r+\sum\limits_{i=1}^r\,\arf(c_{g+i})\equiv0\mod\,m.$$
%Thus
%\begin{align*}
%  0
%  &\equiv(p_1\cdots p_r)\cdot\left((2g-2)+r+\sum\limits_{i=1}^r\,\arf(c_{g+i})\right)\\
%  &\equiv(p_1\cdots p_r)\cdot((2g-2)+r)+\sum\limits_{i=1}^r\,\frac{p_1\cdots p_r}{p_i}\cdot(p_i\arf(c_{g+i}))\\
%  &\equiv(p_1\cdots p_r)\cdot((2g-2)+r)+\sum\limits_{i=1}^r\,\frac{p_1\cdots p_r}{p_i}\cdot(-1)\\
%  &\equiv(p_1\cdots p_r)\cdot\left((2g-2)+r-\sum\limits_{i=1}^r\,\frac1{p_i}\right).
%\end{align*}
If $g>1$ and $m$ is odd then according to Lemma~\ref{lem-normalise} there is a standard basis $\{a_1,b_1,\dots,a_g,b_g,c_{g+1},\dots,c_{g+r}\}$ of $\piorbO(P,p)$ such that
$$(\arf(a_1),\arf(b_1),\dots,\arf(a_g),\arf(b_g))=(0,1,1,\dots,1),$$
hence $\de(P,\arf)=0$ by definition.
If~$g=1$ then $\de$ is a divisor of~$m,p_1-1,\dots,p_r-1$ by definition.
If~$g=0$ then $\de=0$ by definition.

\myskip\noindent
Now let us assume that $t=(g,p_1,\dots,p_r,\de)$ satisfies the conditions~(a)-(e).
Let $P$ be a Riemann orbifold of signature $(g:p_1,\dots,p_r)$ and let
$$\{a_1,b_1,\dots,a_g,b_g,c_{g+1},\dots,c_{g+r}\}$$
be a standard basis of~$\piorbO(P,p)$.
According to Corollary~\ref{cor-arf-function-exists} any tuple of $2g$ values in~$\z/m\z$
can be realised as a set of values on~$a_i,b_i$ of an $m$-Arf function on~$\piorbO(P,p)$,
%We want to construct an $m$-Arf function of type~$t$.
%Since $p_i$ is prime with~$m$, we can choose~$n_i\in\z/m\z$ such that $p_i\cdot n_i+1\equiv0\mod\,m$ for~$i=1,\dots,r$.
%Then
%\begin{align*}
%  &(p_1\cdots p_r)\cdot\left((2g-2)+r+\sum\limits_{i=1}^r\,n_i\right)\\
%  &\equiv(p_1\cdots p_r)\cdot((2g-2)+r)+\sum\limits_{i=1}^r\,\frac{p_1\cdots p_r}{p_i}\cdot(p_in_i)\\
%  &\equiv(p_1\cdots p_r)\cdot((2g-2)+r)+\sum\limits_{i=1}^r\,\frac{p_1\cdots p_r}{p_i}\cdot(-1)\\
%  &\equiv(p_1\cdots p_r)\cdot\left((2g-2)+r-\sum\limits_{i=1}^r\,\frac1{p_i}\right)
%  \equiv0\mod\,m.
%\end{align*}
%Since $\gcd(p_i,m)=1$,
%the equality $(p_1\cdots p_r)\cdot\left((2g-2)+r+\sum\limits_{i=1}^r\,n_i\right)\equiv0\mod\,m$
%implies $(2g-2)+r+\sum\limits_{i=1}^r\,n_i\equiv0\mod\,m$, i.e. $\sum\limits_{i=1}^r\,n_i\equiv(2-2g)-r\mod\,m$.
%The equality $\sum\limits_{i=1}^r\,n_i\equiv(2-2g)-r\mod\,m$
%together with Proposition~\ref{arf-prop-liftexist} and Corollary~\ref{cor-arf-def-by-gen} imply that
In particular if~$g>1$ then for any~$\de\in\{0,1\}$ there exists an $m$-Arf function $\arf^{\de}$
such that $(\arf^{\de}(a_1),\arf^{\de}(b_1),\dots,\arf^{\de}(a_g),\arf^{\de}(b_g))=(0,1-\de,1,\dots,1)$
and if~$g=1$ then for any divisor~$\de$ of $m,p_1-1,\dots,p_r-1$ there exists an $m$-Arf function $\arf^{\de}$ such that $(\arf^{\de}(a_1),\arf^{\de}(b_1))=(\de,0)$.

\myskip\noindent
Let $g>1$.
If~$\de=0$ then the equation $\de(\arf^0)=0$ is satisfied by definition.
If~$\de=1$ and $m$ is even, it remains to prove that $\de(\arf^1)=1$.
To this end we recall that $\sum\limits_{i=1}^g(1-\arf(a_i))(1-\arf(b_i))\mod\,2$ is preserved under the Dehn twists
and hence is equal to~$1$ modulo~$2$ for any standard basis.

\myskip\noindent
Now let~$g=1$.
Then $\de(\arf^{\de})=\gcd(m,p_1-1,\dots,p_r-1,\de,0)=\de$ since $\de$ is a divisor of $\gcd(m,p_1-1,\dots,p_r-1)$. 
\end{proof}

\subsection{Teichm\"uller spaces of Fuchsian groups}

We recall the results on the moduli spaces of Fuchsian groups from~\cite{Zieschang:book}.

\myskip
Let $\Ggp$ be the group generated by the elements
$$v=\{a_1,b_1,\dots,a_g,b_g,c_{g+1},\dots,c_{g+r}\}$$
with defining relations
$$\prod\limits_{i=1}^g\,[a_i,b_i]\prod\limits_{i=g+1}^{g+r}\,c_i=1,\quad c_{g+1}^{p_1}=\cdots=c_{g+r}^{p_r}=1.$$

\myskip
We denote by $\tTgp$ the set of monomorphisms $\psi:\Ggp\to\Aut(\hyp)$
such that
$$\psi(v)=\{a_1^{\psi},b_1^{\psi},\dots,a_g^{\psi},b_g^{\psi},c_{g+1}^{\psi},\dots,c_{g+r}^{\psi}\}$$
is a sequential set of signature $(g;p_1,\dots,p_r)$.
Here we assume that $g>1$.

\myskip
The group $\Aut(\hyp)$ acts on $\tTgp$ by conjugation.
We set
$$\Tgp=\tTgp/\Aut(\hyp).$$

\myskip
We parametrise the space $\tTgp$ by the fixed points and shift parameters of the elements of the sequential sets $\psi(v)$.
We use here the following analogue of a version \cite{N1978moduli}, \cite{Nbook} of the Theorem of Fricke and Klein~\cite{FK}:
%For details of the orbifold case see~\cite{Zieschang:book}.

%\cite{Nbook}, Chapter~1, Theorem~4.1
\begin{thm}

\myskip\noindent
The space $\Tgp$ is diffeomorphic to an open domain in
$$\r^{6g-6+2r}$$ 
which is homeomorphic to $\r^{6g-6+2r}$.
\end{thm}

For an element $\psi:\Ggp\to\Aut(\hyp)$ of $\tTgp$ we write
$$\tMod^{\psi}=\tModgp^{\psi}=\{\al\in\Aut(\Ggp)\st\psi\circ\al\in\tTgp\}.$$
One can show that $\tMod^{\psi}$ does not depend on~$\psi$, hence we write $\tMod$ instead of~$\tMod^{\psi}$.
Let $I\tMod$ be the subgroup of all inner automorphisms of $\Ggp$ and let
$$\Modgp=\Mod=\tMod/I\tMod.$$
We now recall the description of the moduli space of Riemann orbifolds

% \cite{Nbook}, Chapter~1, Section~5
\begin{thm}
\label{moduli-of-surfaces}

\myskip\noindent
The group $\Mod=\Modgp$ and the group of homotopy classes of orientation preserving autohomeomorphisms
of the orbifold of signature $(g:p_1,\dots,p_r)$ are naturally isomorphic.
The group $\Modgp$ acts naturally on $\Tgp$ by diffeomorphisms.
This action is discrete.
The quotient set
$$\Tgp/\Modgp$$
can be identified naturally with the moduli space $\Mgp$ of Riemann orbifolds of signature $(g:p_1,\dots,p_r)$.
\end{thm}

\subsection{Connected components of the moduli space}

\label{components}

\begin{mydef}
We denote by~$\Smgp(t)=\Smgp(g,p_1,\dots,p_r,\de)$ the set of all GQHSS of level~$m$ and signature~$(g:p_1,\dots,p_r)$
such that the associated $m$-Arf function is of type~$t=(g,p_1,\dots,p_r,\de)$.
\end{mydef}

\begin{thm}
\label{top-type-comp}
Let $t=(g,p_1,\dots,p_r,\de)$ be a tuple that satisfies the conditions of Theorem~\ref{thm-is-type},
\ie the space $\Smgp(t)$ is not empty.
Then the space $\Smgp(t)$ is homeomorphic to $\Tgp/\Modmgp(t)$, where $\Tgp$ is homeomorphic to~$\r^{6g-6+2r}$
and $\Modmgp(t)$ acts on $\Tgp$ as a subgroup of finite index in the group $\Modgp$.
\end{thm}

\begin{proof}
Let us consider an element~$\psi$ of the space $\Tgp$.
By definition $\psi$ is an homomorphism $\psi:\Ggp\to\Aut(\hyp)$.
To the homomorphism $\psi$ we attach an orbifold $P_{\psi}=\hyp/\psi(\Ggp)$,
a standard basis
$$v_{\psi}=\{a^{\psi}_1,b^{\psi}_1,\dots,a^{\psi}_g,b^{\psi}_g,c^{\psi}_{g+1},\dots,c^{\psi}_{g+r}\}$$
of $\piorb(P_{\psi},p)$
and an $m$-Arf function $\arf_{\psi}$ on this surface given by
\begin{align*}
  &(\arf_{\psi}(a_1^{\psi}),\arf_{\psi}(b_1^{\psi}))=(\de,0)\quad\text{if}\quad g=1,\\
  &(\arf_{\psi}(a_1^{\psi}),\arf_{\psi}(b_1^{\psi}),\arf_{\psi}(a_2^{\psi}),\arf_{\psi}(b_2^{\psi}),\dots,\arf_{\psi}(a_g^{\psi}),\arf_{\psi}(b_g^{\psi}))\\
  &
  \begin{aligned}
  &=(0,1-\de,1,\dots,1)\quad\text{if}\quad g>1.
  \end{aligned}
\end{align*}
By Theorem~\ref{thm-corresp}, the $m$-Arf function $\arf_{\psi}$ on the orbifold $P_{\psi}$
corresponds to a lift of $\psi(\Ggp)$ into~$\Gm$.
The correspondence $\psi\mapsto \de_{\psi}$ defines a map
$$\Tgp\to\Smgp(t).$$
According to Theorem~\ref{thm-is-type} this map is surjective.
Let $\Modmgp(t)$ be the subgroup of $\Aut(P_{\psi})=\Modgp$ that preserves the $m$-Arf function $\arf_{\psi}$.
For any point in $\Smgp(t)$ its pre-image in $\Tgp$ consists of an orbit of the subgroup~$\Modmgp(t)$.
Thus
$$\Smgp(t)=\Tgp/\Modmgp(t).\qedhere$$
%where $\Modmgp(t)$ is a discrete group.
\end{proof}

\bigskip\noindent
Summarizing the results of Theorems~\ref{thm-is-type} and \ref{top-type-comp} we obtain the following

\bigskip
\begin{thm}
\label{main-thm}

\noindent
\begin{enumerate}[1)]
\item
Two hyperbolic GQHSS are in the same connected component of the space of all hyperbolic GQHSS if and only if they are of the same type.
In other words, the connected components of the space of all hyperbolic GQHSS are those sets $\Smgp(t)$ that are not empty.
\item
The set $\Smgp(t)$ is not empty if and only if $t=(g,p_1,\dots,p_r,\de)$ has the following properties:
\begin{enumerate}[(a)]
\item
The orders~$p_1,\dots,p_r$ are prime with~$m$ and satisfy the condition
$$(p_1\cdots p_r)\cdot\left(\sum\limits_{i=1}^r\,\frac1{p_i}-(2g-2)-r\right)\equiv0\mod m.$$
\item
If $g>1$ and $m$ is odd then $\de=0$.
\item
If $g=1$ then $\de$ is a divisor of~$\gcd(m,p_1-1,\dots,p_r-1)$.
%If $g=1$ then $\de$ is a divisor of~$m$ and~$\gcd(\{j+1\st\num^h_j+\num^p_j\ne0\})$.
\item
If $g=0$ then $\de=0$.
\end{enumerate}
\item
Any connected component~$\Smgp(t)$ of the space of all hyperbolic GQHSS of level~$m$ and signature~$(g:p_1,\dots,p_r)$ is homeomorphic to
$$\r^{6g-6+2r}/\Modmgp(t),$$
where $\Modmgp(t)$ is a subgroup of finite index in the group~$\Modgp$ and acts discretely on $\r^{6g-6+2r}$.
\end{enumerate}
\end{thm}

\begin{rem}
{\bf $\mathbb{Q}$-Gorenstein singularities:}
A normal isolated singularity of dimension at least~$2$ is {\it $\q$-Gorenstein\/}
if there is a natural number~$r$ such that the divisor~$r\cdot\KX$ is defined
on a punctured neighbourhood of the singular point by a function.
Here $\KX$ is the canonical divisor of $X$.
%The smallest such number $r$ is called the {\it index\/} of the singularity.
According to~\cite{Pr:qgor}, hyperbolic $\q$-Gorenstein quasi-homogeneous surface singularities 
are in 1-to-1 correspondence with groups of the form~$C^*\times\lats$,
where $C^*$ is a lift  of a finite cyclic group of order~$r$ into~$\Gm$
and $\lats$ is a lift of a Fuchsian group~$\lat$ into~$\Gm$.
The lift of a finite cyclic group is unique,
hence hyperbolic $\q$-Gorenstein quasi-homogeneous surface singularities are
are in 1-to-1 correspondence with lift of a Fuchsian group into~$\Gm$.
Thus the moduli space of hyperbolic $\q$-Gorenstein quasi-homogeneous surface singularities
coincides with the moduli space of hyperbolic Gorenstein quasi-homogeneous surface singularities
as described in Theorem~\ref{main-thm}.
\end{rem}

\begin{rem}
{\bf Spherical and Euclidean Automorphy Factors:}
For a spherical Gorenstein automorphy factor~$(\cpxprj,\lat,\bdl)$
the group of automorphisms is $\Aut(\pln)=\Aut(\cpxprj)=\PSU(2)$.
The discrete (and hence) finite subgroups of $\PSU(2)$
are the cyclic groups, the dihedral groups and the symmetry groups of the regular polyhedra,
i.e.\ the tetrahedral, octahedral and icosahedral groups.
The corresponding singularities are $A_k$, $D_k$, $E_6$, $E_7$, $E_8$.
For a Euclidean Gorenstein automorphy factor~$(\cpxpln,\lat,\bdl)$
the group~$\lat$ is contained in the translation subgroup of~$\Aut(\cpxpln)$
and can be identified with a sublattice~$\z\cdot 1+\z\cdot\tau$ of the additive group~$\c$,
where~$\tau\in\c$ and~$\Im(\tau)>0$, see~\cite{Dolgachev:1983}.
%The group~$\lat$ can be described
%as a discrete co-compact subgroup of the group of unipotent upper triangular $3\times 3$-matrices.
The corresponding singularities are $\tilde E_6$, $\tilde E_7$, $\tilde E_8$.
All GQHSS other than $A_k$, $D_k$, $E_6$, $E_7$, $E_8$, $\tilde E_6$, $\tilde E_7$, $\tilde E_8$
belong to the class of hyperbolic GQHSS, which is studied in this paper.
\end{rem}

%%\nocite{*}

%%\bibliographystyle{amsalpha}
%%\bibliography{arfgor}

\def\cprime{$'$}
\providecommand{\bysame}{\leavevmode\hbox to3em{\hrulefill}\thinspace}
\providecommand{\MR}{\relax\ifhmode\unskip\space\fi MR }
% \MRhref is called by the amsart/book/proc definition of \MR.
\providecommand{\MRhref}[2]{%
  \href{http://www.ams.org/mathscinet-getitem?mr=#1}{#2}
}
\providecommand{\href}[2]{#2}

\end{document}